\documentclass[a4paper, 11pt, final]{article}
\usepackage[latin1]{inputenc} 
\usepackage[T1]{fontenc}
\usepackage[english]{babel}     
\usepackage{fullpage}
\usepackage{amsmath}
\usepackage{amsthm}
\usepackage{dsfont}\let\mathbb\mathds
\usepackage{setspace}
\usepackage{amsfonts}
\usepackage{amssymb}
\usepackage[all,knot,poly]{xy}
\usepackage{color}
\usepackage{transparent}
\usepackage{graphicx}
\usepackage[top=3cm, bottom=3cm, left=3cm, right=3cm]{geometry}
\everymath{\displaystyle} 
\title{Geometric quantization and semi-classical limits of pairings of TQFT vectors}
\author{Renaud Detcherry \footnote{Universit\'e Paris 6
}}
\date{}
\makeatletter
\renewcommand\@biblabel[1]{}
\makeatother
\newtheorem{theo}{Theorem}[section]
\newtheorem{definition}{Definition}[section]
\newtheorem{lemma}{Lemma}[section]
\newtheorem{proposition}{Proposition}[section]
\DeclareMathOperator{\su}{\textrm{SU}_2}

\DeclareMathOperator{\tr}{\textrm{Tr}}
\DeclareMathOperator{\ad}{\textrm{Ad}}
\begin{document}
\maketitle
\begin{abstract}
This paper exposes an explicit mapping between the TQFT vector spaces $V_{r}(\Sigma)$ defined combinatorially by Blanchet-Habegger-Masbaum and Vogel and spaces of holomorphic sections of complex line bundles on some Kähler manifold, following the approach of geometric quantization. We explain how curve operators in TQFT can then be seen as Toeplitz operators with symbols corresponding to some trace functions. As an application, we show that eigenvectors of these operators are concentrated near the level sets of these trace functions, and obtain asymptotic estimates of pairings of such eigenvectors. This yields an asymptotic for the matrix coefficients of the image of mapping classes by quantum representations.
\end{abstract}
\section{Introduction}
The study of topological quantum field theories (or TQFT) was developed after Witten used the Jones polynomial to heuristically define a collection of invariants of 3-manifolds, cobordisms, surfaces and mapping class on surfaces, satisfying some axioms, including some compatibility with gluings or disjoint union, and gave the expected asymptotic expansion of these 3-manifold invariants, in what is known as the Witten conjecture.
Later, these TQFT were constructed more rigorously by Reshetikhin and Turaev in \cite{RT}, and later by Blanchet, Masbaum, Habegger and Vogel in \cite{BHMV}, in the case where the gauge group is $G=\textrm{SU}_2$, using skein calculus. This second approach, more combinatorial, is the one we use in the following paper. 
\\ To each compact oriented surface $\Sigma$ is associated by the TQFT a sequence of Hermitian vector spaces $V_r(\Sigma)$, parametrized by an integer $r$ called level, to each pants decomposition a basis $(\varphi_{\alpha})$ of $V_r(\Sigma)$, and to each simple closed curve on $\Sigma$ a curve operator $T_r^{\gamma}\in \textrm{End}(V_r(\Sigma))$. The goal of this paper is to compute the asymptotic behavior of pairings $\langle \varphi_{\alpha},\psi_{\beta}\rangle$ of basis vectors corresponding to two pants decomposition when the level goes to infinity.
\\ A helpful tool to study asymptotics of quantum invariants is the theory of geometric quantization. Given a Kähler manifold $(M,\omega, J)$ with $\textrm{dim}(M)=2 n$, we define a prequantization bundle as a complex line bundle with an Hermitian form $h$ that has Chern curvature $\frac{1}{i}\omega$ and a half-form bundle, that is a square root of the bundle of complex $n$-forms. For $L$ a prequantization bundle and $\delta$ a half-form bundle, we have a sequence of (finite dimensional when $M$ is compact) vector spaces $H^0 (M,L^r\otimes \delta)$: the spaces of holomorphic sections of $L^r\otimes \delta$. 
\\ A natural candidate to present the vector space $V_r(\Sigma)$ and curve operators $T_r^{\gamma}$ as arising from the geometric quantization of some Kähler manifold and function is the moduli space $\mathcal{M}(\Sigma)=\textrm{Hom}(\pi_1\Sigma ,\su)/\su$ of representations of the fundamental group of $\Sigma$ in $\su$ modulo conjugation. This space has a natural symplectic form on it, defined by Atiyah and Bott in \cite{AB82}.
\\ Also, in the setting of geometric quantization, each smooth integrable function $f$ on $\mathcal{M}(\Sigma)$ is associated with a sequence of endomorphisms of $H^{0}(M,L^r \otimes \delta)$ called a Toeplitz operator of symbol $f$. Curve operators $T_r^{\gamma}$ will be represented as Toeplitz operators with principal symbol the trace functions $f_{\gamma}(\rho)=-\textrm{Tr}(\rho(\gamma))$ which are continuous functions on $\mathcal{M}(\Sigma)$.
\\ Then, results of microlocal analysis state that the joint eigenvectors of such Toeplitz operators concentrate on the level sets of their principal symbol. Thus, sequence of vectors $\varphi_{\alpha_r}$ with $\frac{\alpha_r}{r} \underset{r \rightarrow +\infty}{\rightarrow} x$ should carry most of their mass on a neighborhood of the set $\Lambda_{x}=\lbrace \rho \, / \, \tr(\rho(C_i))=-2 \cos(\pi x_i) \rbrace$.
\\ The final goal of this paper is to compute the asymptotic expansion of pairings $\langle \varphi_{\alpha} ,\psi_{\beta} \rangle$ of basis vectors of $V_r(\Sigma)$ corresponding to two pants decompositions of $\Sigma$. There are two ways of thinking of these pairings: if one decomposition is the image of the other by an element of the mapping class group $\Gamma_g$ of $\Sigma$, what we compute is a limit of matrix coefficients of the quantum representations of $\Gamma_g$ on $V_r(\Sigma)$.
\\ Alternatively, the pairings can be viewed as special Reshethikhin-Turaev invariants: choose two handlebodies, one corresponding to each pants decomposition, insert a trivalent colored graph in each with colorings $\alpha$ and $\beta$, and glue them together to obtain a 3-manifold with 2 trivalent-colored graphs inside, which represent some linear combination of links. The pairing $\langle \varphi_{\alpha},\psi_{\beta} \rangle$ is just the Reshetikhin-Turaev invariant of this manifold with links.
\\ The vectors $\varphi_{\alpha_r}$ and $\psi_{\beta_r}$ concentrate on $\Lambda_{\frac{\alpha_r}{r}}$ and $\Lambda_{\frac{\beta_r}{r}}'$ where $\Lambda_y '=\lbrace \rho \, / \, \rho(D_j)=-2 \cos(\pi y_j) \rbrace$ when $r \rightarrow \infty$. Under some condition of genericity, the Lagrangian $\Lambda_x$ and $\Lambda_y '$ have a transverse intersection, consisting only of a finite number of points, and we show that the pairing $\langle \varphi_{\alpha_r},\psi_{\beta_r} \rangle$ has an asymptotic expansion consisting of a sum of contributions of these points as follows:
$$ \langle \varphi_{\alpha_r} ,\psi_{\beta_r} \rangle = u_r (\frac{r}{2 \pi})^{-\frac{n}{2}} \frac{1}{\sqrt{V_x V_y '}}\underset{z \in \Lambda_{\frac{\alpha_r}{r}} \cap \Lambda_{\frac{\beta_r}{r}} '}{\sum} \frac{e^{i r \eta (z_0 , z)}}{ |\det(\lbrace \mu_i , \mu_j ' \rbrace)|^{\frac{1}{2}}}i^{m(z_0,z)} +O(r^{-\frac{n}{2}-1})$$
where we set $\textrm{dim}(\mathcal{M}(\Sigma))=2 n$ (that is $n=3 g -3$), $u_r$ is a sequence of complex numbers of moduli $1$, the functions $\mu_i=-\tr (\rho (C_i))$ and $\mu_j '=-\tr (\rho (D_j))$ are the principal symbols of $T_r^{C_i}$ and $T_r^{D_j}$, $\lbrace \cdot , \cdot \rbrace$ is the Poisson bracket in $\mathcal{M}(\Sigma)$, $V_x$ (resp $V_y '$) is the volume of $\Lambda_x$ (resp. $\Lambda_y '$) for the volume form $\beta=d \theta_1 \wedge \ldots \wedge d \theta_n$ (resp. $\beta ' =d \theta_1 ' \wedge \ldots d \theta_n '$) where $d \theta_i$ (resp. $d \theta_i '$) are basis of $T^* \Lambda_x$ (resp. $T^* \Lambda_y '$) dual to the Hamiltonian vector fields $X_i$ of $\mu_i$ (resp. $X_i '$ of $\mu_i'$). 
\\ Moreover, $z_0$ is some specific point in the finite intersection, and for $z \in \Lambda_x \cap \Lambda_y '$, if we choose $\gamma_{z_0,z}$ a loop consisting of a path from $z_0$ to $z$ in $\Lambda_x$ and a path from $z$ to $z_0$ in $\Lambda_y '$, then $\eta (z_0 ,z)$ is the holonomy of the prequantization bundle $L$ along $\gamma_{z_0 ,z}$. The line bundle $L$ having curvature $\frac{1}{i}\omega$, the quantity $\eta(z_0 , z)$ can be defined alternatively as follows: take an oriented disk $D(z_0 ,z)$ in $\mathcal{M}(\Sigma)$ whose boundary is the loop $\gamma_{z_0 ,z}$. Then $\eta(z_0, z)$ is its symplectic area: $\eta (z_0 ,z)=\int_{D(z_0 ,z)} \omega$  
\\ Finally, $m(z_0 ,z) \in \mathbb{Z}$ and is some kind of Maslov index, the computation of which is explained in Section \ref{sec:pairing}. A path from $z_0$ to $z$ along $\Lambda_x$ induces a path in the oriented Lagrangian Grassmanian $LG^+ (\mathcal{M}(\Sigma))$, similarly for the path from $z$ to $z_0$ in $\Lambda_y '$. We connect these to get a loop in the oriented Lagrangian Grassmanian by turning "positively" in the oriented Lagrangian Grassmanian of $T_z \mathcal{M}(\Sigma)$ and $T_{z_0} \mathcal{M}(\Sigma)$. The index $m(z_0 ,z)$ ought to correspond the homotopy class of this loop in $\pi_1 LG^+ (\mathcal{M}(\Sigma))=\mathbb{Z}$.
\\ Notice that the definition of $\eta(z_0 ,z)$ depends only on the homotopy class of the loop $\gamma_{z_0 ,z}$ as the line bundle $L$ is a flat bundle on the Lagrangian $\Lambda_x$ and $\Lambda_y '$. Actually, the quantity $r \eta (z_0 ,z) + \frac{\pi}{2}m(z_0 ,z)$ will be also independent of this homotopy class modulo $2 \pi \mathbb{Z}$, as a result of Bohr-Sommerfeld conditions. 
\\ Finally, it is remarkable that the first term of the asymptotic expansion of this pairing is defined uniquely from the positions of the Lagrangian $\Lambda_x$ and $\Lambda_y '$ and the symplectic structure of $\mathcal{M}(\Sigma)$, though the machinery of geometric quantization introduced also a Kahler structure. It had to be expected as the quantum invariants of which we computed an asymptotic expansion are topological invariants, and do not depend on the complex structure.
\\
\\ The idea of linking curve operators in TQFT to Toeplitz operators originates in the work of Andersen \cite{And06}. Andersen works in the geometrical viewpoint of TQFT, reprensenting TQFT vector spaces $V_r(\Sigma)$ as spaces $V_r^{\sigma}(\Sigma)$ of holomorphic sections on the moduli space depending on the choice of a complex structure $\sigma$ on $\Sigma$, and introduces in \cite{And06} some Toeplitz operators with trace functions as principal symbols, and shows that they approximate curve operators at first order. This approach proved rapidly fruitful: Andersen was able to use these Toeplitz operators to derive the asymptotic faithfulness of the quantum representations of the mapping class group \cite{And06}, as well as other results \cite{And08}\cite{And09}\cite{And10}.
\\ The geometrical viewpoint of TQFT makes an heavy use of the Hitchin connection to relate the $V_r^{\sigma}(\Sigma)$ for different choices of $\sigma$, making the computations quite unexplicit. Thereas, in this paper, as we using only the skein-theoretic presentation of TQFT, we are able to present a simple and explicit isomorphism between $V_r(\Sigma)$ and holomorphic sections on some prequantized manifold $M$. The drawback is that $M$ is a manifold with boundary, whereas Andersen is able to work on smooth closed manifolds.
\\
\\ The present paper takes inspiration in the work of Paul and Marché \cite{MP}, in which they showed an asymptotic formula for curve operators on the pointed torus and the four-holed sphere, then presented curve operators on these two surfaces as Toeplitz operators with trace functions as principal symbols, and deduced from such a presentation the asymptotics of quantum $6j$-symbols and coefficients of the pointed $S$-matrix. 
\\ They also conjectured in \cite{MP} that their results would apply for arbitrary compact oriented surfaces. The author devoted the paper \cite{Det} to the generalization of the first part of their work, proving the conjectured asymptotic formula for curve operators on general surfaces. We will recall this result in Section \ref{sec:tqft} as it will be needed in further sections.
\\ Finally, after presenting in this paper curve operators on arbitrary surfaces $\Sigma$ as Toeplitz operators, we will apply this discussion to compute the asymptotics of some quantum invariants.
\\ 
\\ Since their discovery the asymptotic behavior of the Witten-Reshetikhin-Turaev invariants $Z_r(M)$ of 3-manifolds $M$ has been a big object of interest. Using his path integral description of invariants, Witten conjectured an asymptotic expansion for $Z_r(M)$. The formula for this expansion, first mentioned in \cite{Jef92} is:
$$Z_r(M)=(1+O(r^{-1}))\underset{\rho \in \mathrm{Hom}(\pi_1 M ,\su)/\su}{\sum} e^{2 i \pi r CS(\rho)} r^{\frac{h^1(\rho))-h^0 (\rho))}{2}} \sqrt{\mathrm{Tor}(M,\rho)}I_{\rho}$$ 
where $CS$ is the Chern-Simons functional, $h^i(\rho)=\dim(H^i(M,\mathrm{Ad}(\rho))$ are twisted cohomology groups of $M$ for the adjoint representation of $\rho$, $\mathrm{Tor}$ is the Reidemeister torsion and $I_{\rho}$ is an root of unity of order $8$, which can be computed using a spectral flow.
\\
\\ The semi-classical technics of this paper used to compute quantum invariants by counting contributions of intersections of Lagrangians follow the spirit of a series of two papers by Charles and Marché \cite{CM11a}\cite{CM11b}. 
Their paper establishes that the Reshetikhin-Turaev invariants of Dehn fillings of the figure-eight knot satisfy the asymptotic expansion conjectured by Witten.
They used semi-classical analysis: the complement of the figure-eight knot, associated by the TQFT to a vector (or "knot state" of the figure-eight knot) in $V_r(\mathbb{T}^2)$ where $\mathbb{T}^2$ is the peripheral torus of the figure-eight knot.
The space $V_r(\mathbb{T}^2)$ was reinterpreted as a space of holomorphic sections on $\mathcal{M}(\mathbb{T}^2)$ and the knot state was shown to concentrate on the character variety of the figure-eight knot complement, allowing them to compute the Reshetikhin-Turaev invariants of Dehn fillings by adding contributions of intersections points of this character variety with some Lagrangian of $\mathcal{M}(\mathbb{T}^2)$.
\\ Other examples of 3-manifolds satisfying the Witten conjecture include many Seifert manifolds, as shown by various authors \cite{Jef92}\cite{Roz96}\cite{LZ99}\cite{Hik05}\cite{HT01}, finite order mapping tori as shown by Andersen \cite{And11} or mapping tori satisfying a transversality condition by a paper of Charles \cite{Cha10b}.
\\ The limits of the papers \cite{CM11a} and \cite{CM11b} is that explicit formulas of colored Jones polynomial of the figure-eight knot are used, making their approach difficult to generalize to arbitrary knots.
Similarly, most other proofs of specific cases of the Witten conjecture also use some explicit computations, and thus have only a narrow range of applications. 
\\
\\ Our asymptotic formula is in some sense a generalization of the formula of Witten: we give the asymptotic expansion of quantum invariants of gluings of two handlebodies with trivalent colored graphs in it. The Witten conjecture correspond to the case of trivial colorings of the trivalent graphs: if we choose a Heegaard decomposition $H \underset{\Sigma}{\cup} H'$ of $M$, the sum over representations $\pi_1 M \rightarrow \su$ is a sum over the intersection points of the sets of representations $\pi_1 \Sigma \rightarrow \su$ which can be extended to $\pi_1 H$ and $\pi_1 H'$ respectively. The Chern-Simons invariants is analog to our functionnal $\eta$, and the Reidemeister torsion to our determinant of the matrix of Poisson bracket. Finally $I_{\rho}$ looks like the Maslov index in our formula.
\\ However, the proof of our formula fails in the case of trivial colorings of the trivalent graphs: the Lagrangians of which we consider the intersections have always intersection in the boundary of $M$, corresponding to intersections in the singular part of $\mathcal{M}(\Sigma)$. Furthermore, proper vectors of Toeplitz operators at critical level of the principal symbol are not well understood.
\\ It is nonetheless possible that this approach could work if we had a deeper understanding of the singularities of $\mathcal{M}(\Sigma)$ and the process of geometric quantization in a singular setting.
\\
\\ \textbf{Acknowledgments:} The author would like to thank Julien Marché and Laurent Charles for many helpful discussions and for their constant support.

\section{Overview of the moduli space $\mathcal{M}(\Sigma)$}
\subsection{Pants decomposition and Hamiltonian torus action}
\label{sec:moduli}
Let $\Sigma$ be a closed compact oriented surface. We write $\mathcal{M}(\Sigma)=\textrm{Hom}(\pi_1 \Sigma , \su) /\textrm{SU}_2$ for the moduli space of representations of the fundamental group of $\Sigma$ in $\textrm{SU}_2$ modulo conjugation. It can be shown that $\mathcal{M}(\Sigma)$ is then a real algebraic variety.
\\ Alternatively, this space can be described as the space of $\textrm{su}_2$-connections $A \in \Omega^{1}(\Sigma ,\textrm{su}_2)$ , such that $dA +\frac{1}{2}[A \wedge A] =0$ (flatness condition), modulo the gauge action by $G=C^{\infty}(\Sigma,\textrm{SU}_2)$.
\\ The gauge action is given by $A^{g}=g Ag^{-1} +g^{-1} d g$ for any $g \in G$.
\\ There is a partition of $\mathcal{M}(\Sigma)$ given by the set of conjugacy classes of irreducible representations $\mathcal{M}^{irr}(\Sigma)$ and the set $\mathcal{M}^{ab}(\Sigma)$ of conjugacy classes of abelian representations. When the surface $\Sigma$ is of genus $g \geq 2$, the algebraic variety $\mathcal{M}(\Sigma)$ is smooth at $[\rho]$ if and only if $\rho$ is an irreducible representation. 
\\If the connection $A$ represents $\rho$, then the tangent space $T_{\rho} \mathcal{M}(\Sigma)$ at $A$ is then given by all 1-forms $\alpha \in \Omega^{1}(\Sigma ,\textrm{su}_2)$, such that $d \alpha +[\alpha \wedge A] =0$, modulo gauge action by $\Omega^0 (\Sigma,\textrm{su}_2)$ acting by translating $\alpha$ by $d\xi +[A ,\xi]$, for any $\xi \in \Omega^0 (\Sigma ,\textrm{su}_2)$. 
\\Furthermore, a natural symplectic structure on $\mathcal{M}^{irr}(\Sigma)$ was introduced by Atiyah and Bott \cite{AB82} and then Goldman \cite{Gol}. This symplectic structure depends on a choice of normalization: for $\alpha$ and $\beta \in T_{A}\mathcal{M}(\Sigma)$ we choose the normalization:
$$\omega_A (\alpha ,\beta)= \frac{1}{2 \pi}\int_{\Sigma} \textrm{Tr}(\alpha \wedge \beta)$$
A pants decomposition of $\Sigma$ is a family of simple closed curves $\mathcal{C}=\lbrace C_e \rbrace_{e \in E}$ that separate $\Sigma$ into a disjoint union of three-holed sphere. We write $S$ for the set of triples $(e,f,g) \in E^3$ such that $C_e$, $C_f$, and $C_g$ bound a pair of pants in the decomposition. This data gives rise to an Hamiltonian torus action on $\mathcal{M}(\Sigma)$ by a torus of dimension $|E|$. Such an action is characterized by its momentum mapping:  
$$\begin{array}{cccc}
\mu : & \mathcal{M}(\Sigma) & \rightarrow & \mathbb{R}^E \\
  & \rho & \rightarrow &  h_{\mathcal{C}}(\rho)
\end{array}$$

where the application $h_{\mathcal{C}}$ is given by its components:
$$ h_{C_e}(\rho)=\frac{1}{\pi}\textrm{Acos}(\textrm{Tr}(\rho(C_e)))$$
which can be shown \cite{Gol} to be Poisson commuting functions on $\mathcal{M}(\Sigma)$.
\\ The Hamiltonian flows of these Poisson commuting functions give the action:
$$\begin{array}{ccc}\mathbb{R}^{E} & \rightarrow & \mathcal{M}(\Sigma) \\
    (\theta_e)_{e \in E} &\rightarrow & \theta \cdot \rho 
\end{array}$$
 The image of the momentum mapping $\mu$ is a polytope $P$ inside $\mathbb{R}^E$. The polytope $P$ consists of all $(x_e)_{e \in E} \in \mathbb{R}^{E}$ such that, if $(e,f,g) \in S$, then:
\begin{itemize}
  \item[(i)] $|x_f -x_g| \leq x_e \leq x_f +x_g$
 \item[(ii)] $x_e+x_f+x_g \leq 2$
\end{itemize} 
The momentum mapping $\mu$, and the associated Hamiltonian flows were described by Jeffrey and Weistmann \cite{JW}, see also \cite{Gol}. 
\\Given a choice of orientation of the curves $C_e$, the Hamiltonian action $\mathbb{R}^E \circlearrowright \mathcal{M}(\Sigma)$ can actually be lifted to an action on $\textrm{Hom}(\pi_1(\Sigma),\textrm{SU}_2)$ which acts on representations as follows: we pick a point of $C_e$ as base point of $\pi_1 \Sigma$. We can also assume up to conjugation that $\rho (C_e)$ is diagonal. Any element of $\pi_1 \Sigma$ is a product of loops intersecting $C_e$ at most once.
\\ If the curve $C_e$ is nonseparating the image of such a loop $\gamma$ by the representation $\theta_e \cdot \rho$ is $\rho(\gamma)$ if $\gamma$ has zero algebraic intersection with $C_e$.
\\ If the algebraic intersection of $\gamma$ and $C_e$ is one we set $(\theta_e \cdot \rho)(\gamma)=U_{\theta_e} \rho(\gamma)$ where
$U_{\theta_e}$ is the matrix $\begin{pmatrix} e^{ i \theta_e} & 0 \\ 
                0 & e^{- i \theta_e}
\end{pmatrix}$.
\\ These conventions suffice to define a representation $\theta_e \cdot \rho$ so that the corresponding action on $\mathrm{Hom}(\pi_1 \Sigma , \su)$ lifts the Hamiltonian action of $h_{C_e}$ on $\mathcal{M}(\Sigma)$. If the curve $C_e$ is separating, we have to conjugate some of these holonomies, see \cite{Gol} for details.
\\ In \cite{JW} the kernel of the action on $\mathcal{M}(\Sigma)$ was computed and shown to be some explicit lattice $2 \pi \Lambda$ in $\mathbb{R}^E$. For $e \in E$ let $u_e$ be the vector in $\mathbb{R}^E$ such that all components of $u_e$ vanish expect the $e$-th component which is $1$. For $v=(e,f,g) \in S$ we also introduce the vector $u_v =\frac{u_e+u_f+u_g}{2}\in \mathbb{R}^E$. Note that the same label can appear twice in $v=(e,f,g)$. 
\\ Then [JW-prop 5.2] shows that:
$$ \Lambda = \textrm{Vect}_{\mathbb{Z}}\lbrace (u_e)_{e \in E} , \  (u_v)_{v \in S} \rbrace $$
and furthermore the action of $T=\mathbb{R}^{E}/2 \pi \Lambda$ is free on $\mu^{-1}(\mathring{P})$. 
\\ Now, suppose we set $\omega_P =\sum dx_i \wedge d \theta_i $ on $\mathring{P}\times T$. Then $\omega_P$ is a symplectic form on $\mathring{P}\times T$.
\\Given a Lagrangian section $s : P \rightarrow \mathcal{M}(\Sigma)$ of the momentum map, the map:
$$\begin{array}{ccccc} \rho & : & \mathring{P} \times T & \longrightarrow  & \mathcal{M}(\Sigma) \\
                       &   & (x,\theta) & \longrightarrow & \rho_{x,\theta}=\theta \cdot s(x)
\end{array}$$
 maps $\mathring{P} \times T$ into the open subset $\mu^{-1}(\mathring{P})$ of $\mathcal{M}(\Sigma))$. As $\mu^{-1}(\partial P)$ is a strict real algebraic subvariety of $\mathcal{M}(\Sigma)$, the image $\mu^{-1}(\mathring{P})$ is dense in $\mathcal{M}(\Sigma)$. The condition that $s$ is an Lagrangian section of $\mu$ ensures that $\mu(\rho_{x,\theta})=x$ and that the parametrization sends the $2$-form $\underset{e}{\sum} d x_e \wedge d \theta_e$ on $\mathring{P} \times T$ to the symplectic form $\omega$ on $\mathcal{M}(\Sigma)$. 
\\ Such a parametrization is called an action-angle parametrization of $\mu^{-1}(\mathring{P})$, the $x$-coordinates are called action coordinates and $\theta$-coordinates are angle coordinates. 
 
\section{TQFT and geometric quantization}
This section is devoted to the definition of TQFT spaces $V_r(\Sigma)$ associated to each level $r \in \mathbb{N}^*$ and each closed oriented surface $\Sigma$, as well as the definition of the curve operators acting on these spaces. We begin by a quick overview of the combinatorial framework for TQFT of \cite{BHMV}, then we rebuild these objects in a more analytic framework in Subsection 2.2.  
\subsection{TQFT spaces and curve operators}
\label{sec:tqft}
For $M$ a 3-manifold and $A \in \mathbb{C}$, we will denote by $K(M,A)$ the Kauffman module of $M$ at $A$. The well-known combinatorial construction of Witten-Reshetikhin-Turaev TQFT by \cite{BHMV} associates to each compact oriented surface $\Sigma_g$ of genus $g$ and each $r>0$ is associated a vector space $V_r(\Sigma_g)$. This vector space is obtained as a quotient of the Kauffman module $K(H_g ,\zeta_r)$ where $\zeta_r=-e^{i \frac{\pi}{2 r}}$ of $H_g$, a handle body of boundary $\Sigma_g$. More precisely, it is the quotient of this vector space by all negligible elements:
\ if $L \in K(H_g ,\zeta_r)$ and $L' \in K(H_g ' ,\zeta_r)$ where $H_g \cup H_g ' =S^3$ we have a pairing $\langle L ,L' \rangle \in K(S^3 ,\zeta_r)$. We call $L$ a negligible element if this pairing vanishes for all $L' \in K(H_g ' ,\zeta_r)$. Let $N_r$ the space of negligible elements in $K(H_g,\zeta_r)$, we then have
$$V_r(\Sigma_g)=K(H_g,\zeta_r)/N_r$$
Though our definition may seem like it depends on the choice of an handlebody $H_g$, it follows from \cite{BHMV} that the dimension of $V_r(\Sigma_g)$ is independent of this choice.
\\
With this definition, we see that $K(\Sigma,\zeta_r)$ the Kauffman algebra of $\Sigma$ acts on $V_r(\Sigma)$: indeed $z \in K(\Sigma ,\zeta_r)$ acts on $h \in K(H_g,\zeta_r)$ by stacking $z$ over $h$ to obtain an element of $K(H_g,\zeta_r)$ and it is easy to see that this action maps negligible elements to negligible elements. Given a simple closed curve $\gamma$ on $\Sigma$, we call the induced operator on $V_r(\Sigma)$ the curve operator $T_r^{\gamma}$.
\\ According to \cite{BHMV}, $V_r(\Sigma)$ has a natural Hermitian structure. Furthermore, given a pants decomposition by curves $(C_e)_{e \in E}$ of $\Sigma$, we choose a trivalent banded graph $\Gamma$ embedded in $\Sigma$ with the following properties: First the graph $\Gamma$ has one trivalent vertex in each pants of the decomposition. Secondly we require $\Gamma$ to have one edge $e$ for each curve $C_e$ of the decomposition with the additional property that the edge $e$ cuts $C_e$ exactly once and joins the vertex corresponding to the pans on each side of the curve $C_e$. We will call such a graph a dual graph to the pants decomposition of $\Sigma$.
\\ Then \cite{BHMV} gave an orthonormal basis of $V_r(\Sigma)$ as follows: the basis $(\varphi_c)$ is indexed by $r$-admissible colorings of the edges of $\Gamma$.
\\ An $r$-admissible coloring of $\Gamma$ is an application $c \ : \ E \rightarrow \mathbb{N}$ such that $\forall (e,f,g) \in S$
\begin{itemize} 
\item  $c_e+c_f+c_g <2 r$
\item  $c_e+c_f+c_g$ is odd 
\item  $|c_e-c_f| <c_g <c_e+c_f$
\end{itemize}
Note that the conditions above differ slightly from that of \cite{BHMV}: we shifted all colors by one, which will be convenient later. With this conditions, we must have $c_e \in \lbrace 1,2, \ldots ,r-1 \rbrace$ for all $e \in E$. The vectors $\varphi_c$ are of norm $1$, and are obtained as a specific combination of links in $H_g$, see \cite{Det} for details.
Furthermore, if $c$ is not an $r$-admissible coloring, by convention, we set $\varphi_c=0$. We will denote by $I_r$ the set of admissible colorings. We also denote by $I_{\infty}$ the set of $c \, : \, E \rightarrow \mathbb{N}$ satisfying the last two of the three conditions above.
\\
Finally, we will use the following identity which describes the asymptotic behavior of curve operators:
\begin{theo} \label{curvop}\cite{Det} Let $\gamma$ be a simple closed curve on $\Sigma$ and let $M_e =\sharp (\gamma \cap C_e)$. We also suppose that $\Gamma$ is a planar graph (that is the ribbon-graph $\Gamma$ can be embedded in the plane). Then there exists functions $F_k^{\gamma}$ indexed by $k \ : \ E \rightarrow \mathbb{Z}$, such that
\begin{itemize} 
\item[-]$F_k^{\gamma}$ is analytic on $V_{\gamma}=\lbrace (x,h) \in \mathbb{C}^E \times \mathbb{R}^+ \ / \ (\textrm{Re}(x_e) +\varepsilon_e M_e h) \in \mathring{P} , \ \forall \varepsilon \in \lbrace \pm 1 \rbrace^E \rbrace$
\item[-]$F_k^{\gamma}=0$ if there exists $e \in E$ such that $|k_e|>M_e$
\item[-]For any $r$-admissible coloring $c$, we have :
$$T_r^{\gamma} \varphi_c= \underset{k :E \rightarrow \mathbb{Z}}{\sum} F_k^{\gamma}(\frac{c}{r},\frac{1}{r})\varphi_{c+k}$$
\item[-]If, for $(\tau,h) \in V_{\gamma}$ and $\theta \in \mathbb{R}^E/\Lambda$, we set $$\sigma^{\gamma}(\tau ,\theta ,h)=\underset{k:E \rightarrow \mathbb{Z}}{\sum} F_k^{\gamma}(\tau,h)e^{i k \theta}$$
then we have the following asymptotic expansion:
$$\sigma^{\gamma}(\tau ,\theta ,h)=-\textrm{Tr}(\rho_{\tau ,\theta}(\gamma)) +\frac{h}{2 i}\underset{e \in E}{\sum} \frac{\partial^2}{\partial \tau_e \partial \theta_e}(-\textrm{Tr}(\rho_{\tau ,\theta}(\gamma))+O(h^2)$$
where $\rho_{\tau ,\theta}$ is a parametrization of $h_{\mathcal{C}}^{-1}(\mathring{P})$ by action-angle coordinates. The $O(h^2)$ is uniform on compact subsets of $\mathring{P} \times \mathbb{R}^E /\Lambda$
\end{itemize}
\end{theo}
Note that an action-angle coordinate as defined in Section \ref{sec:moduli} is unique only up to a shift in angle coordinates. The paper \cite{Det} explains exactly what action-angle parametrization has to be chosen, but we will not need it here.
\\
When $\Gamma$ is not a planar graph, these formulas are shifted by signs using relative spin structures on $(\Gamma,\partial \Gamma)$, see \cite{Det}. In the remaining of the paper, we will always consider pants decompositions that have a planar dual graph. It is easy to construct a pants decomposition of each surface of genus $g$ that has a planar dual graph, and applying the action of the mapping class group, we can construct many others such decomposition.
\\ This asymptotic expansion for the matrix coefficient was first remarked and proved by Marché and Paul in \cite{MP} in the special cases of the four-holed sphere and the one-holed torus, while the general result for arbitrary compact oriented surface $\Sigma$ was enonced and proven by the author in \cite{Det} . The proof used fusion rules, the description of the Kauffman algebra as a deformation algebra of the algebra of regular functions on $\mathcal{M}(\Sigma)$ and the algebraic properties of curve operators. The spirit of the next section is to use this formulas to view the curve operator $T_r^{\gamma}$ associated to a curve $\gamma$ on $\Sigma$ as a Toeplitz operator with principal symbol the trace function $\sigma^{\gamma}(\tau ,\theta)=-\textrm{Tr}(\rho_{\tau,\theta}(\gamma))$ which is a function on the subset $\mu^{-1}(\mathring{P})$ of $\mathcal{M}(\Sigma)$.
\subsection{TQFT vector spaces $V_r(\Sigma)$ as spaces of holomorphic sections of line bundles}
\label{sec:holosec}
We want now to translate the combinatorial definition of the TQFT space of Subsection 2.1 in an analytic framework, and see the $V_r(\Sigma)$ as spaces of holomorphic $\mathcal{L}^2$ sections of a complex line bundle over a Kähler manifold. 
\\ Since the discovery of Witten-Reshetikhin-Turaev TQFT, it has been a popular endeavor to link the combinatorial definition of TQFT with a definition based on geometric quantization. Given a compact Kähler manifold $M$, with a prequantization line bundle $L$ (that is a line bundle with Chern curvature $\frac{1}{i}\omega$) and a half-form bundle, we have a sequence of  vector spaces $V_r=H^{0}(M ,L^r\otimes \delta)$, and any continuous function $f$ on $M$ gives rise to a sequence of operators $\mathcal{T}_r^{f}=\Pi_r m_f $ where $\Pi_r$ is the orthogonal projector from $\mathcal{L}^2$ sections of $L^r \otimes \delta$ to the space of holomorphic sections. 
\\ The natural geometric object to represent the combinatorial TQFT spaces is then the moduli space $\mathcal{M}(\Sigma)$ together with its Chern-Simons bundle, and a half-form bundle. The problem is that for a general genus $g$, $\mathcal{M}(\Sigma)$ is not smooth, and also there is no canonical choice of complex structure to work with. These hurdles can be solved: for each complex structure $\sigma$ on $\Sigma$ the geometric quantization process yields TQFT spaces $V_r^{\sigma}(\Sigma)$ and there is a connection on the Teichmüller space of $\Sigma$ called the Hitchin connection, which gives a way of identifying the various $V_r^{\sigma}(\Sigma)$ arising from different complex structures on $\mathcal{M}(\Sigma)$ (see \cite{H}). The non-smoothness of $\mathcal{M}(\Sigma)$ is usually avoided of by working with the moduli space of $\Sigma$ with a puncture and choosing appropriate holonomy around the puncture  instead of the moduli space of $\Sigma$. 
\\ It has been showed in \cite{AU} that the TQFT defined by the geometric approach is isomorphic to the combinatorial one. However, the geometric approach to TQFT loses some of the structure of the combinatorial approach: it is not clear how to geometrically define the Hermitian structures on the vector spaces $V_r(\Sigma)$, or the natural basis associated to pants decompositions, also, the identification is quite unexplicit. 
\\ As we wish to use analysis to study pairings of such basis vectors, we will take another approach. Instead of using $\mathcal{M}(\Sigma)$ to do geometric quantization procedures, we will use some open subset of $\mathcal{M}(\Sigma)$ associated to a pants decomposition of $\Sigma$: the set of regular points of the momentum map associated to the pants decomposition. Then we are able to very easily define a complex structure on this set, and to exhibit an isomorphism between $V_r(\Sigma)$ and a space of holomorphic sections over this open subset. 
\\
\\  Given a pants decomposition of $\Sigma$, the set $\mu^{-1}(\mathring{P})$ of regular points of the associated momentum map is an open dense subset of $\mathcal{M}(\Sigma)$ as described in Section \ref{sec:moduli}. By action-angle coordinates it is symplectomorphic to $\mathring{P} \times T$ equiped with the symplectic form $\omega=\sum dt_i \wedge d \theta_i$ . This open set can be seen as a submanifold of $M=\mathbb{R}^E \times T$ which we will equip with a Kähler structure.
\\ A symplectic form on $M$ is given by the formula $\omega=\sum dt_i \wedge d \theta_i$. The complex structure on $M$ will be induced by the map
$$\begin{array}{rlcl}
Z:&\mathbb{R}^E \times T & \rightarrow & (\mathbb{C}^*)^E/\Upsilon \\
 &(t_i , \theta_i) & \rightarrow & z_i=e^{\frac{t_i}{2}+i \theta_i}
\end{array}$$ and the usual complex structure of $(\mathbb{C}^*)^E$.
\\Here $\Upsilon$ is the discrete subgroup of $(\mathbb{C}^*)^E$ generated by $\lbrace \varepsilon^v \ , v \in S \rbrace$ where if $v=(e,f,g)$, then $\varepsilon^v_k=(-1)^{\delta_{e k}+\delta_{f k}+\delta_{g k}}$, where $\delta_{i j}$ is the Kronecker sympbol. In fact, $(\mathbb{C}^*)^E/\Upsilon$ is just isomorphic to $(\mathbb{C}^*)^E$, but it is easier to work with these coordinates to give an expression of the symplectic form $\omega$. Note that the symplectic form induced on $(\mathbb{C}^*)^E/\Upsilon$ by this map is the form $i\sum d w_j \wedge d \overline{w_j}$, where $w_j=\ln (z_j)$, for a local determination of the logarithm.
\\ We endow the line bundle $M \times \mathbb{C}$ with the Hermitian form $h$ such that $h(t,\theta)(1)=e^{-\varphi}$ where $\varphi=\frac{||t||^2}{2}=\frac{1}{2}\sum \ln(z_i \overline{z_i})^2$. The Chern curvature of this complex line bundle is $\partial \overline{\partial}\varphi= \sum d w_j \wedge d\overline{w_j}=\frac{1}{i}\omega$. That is, this is a prequantization bundle.
\\
\\ The manifold $M$ also carries a half-form bundle $\delta$: the bundle of $n$-form is trivial as the $n$-form $\frac{d z_1}{z_1}\wedge \ldots \wedge \frac{d z_n}{z_n}$ is well-defined globally (because the action of $\Upsilon$ leaves each $\frac{d z_i}{z_i}$ invariant). The square root of the $n$-form bundle are then parametrized by $H^1(M,\lbrace \pm 1 \rbrace )$. We then choose as half-form bundle a flat bundle with holonomy $-1$ along the loops $(t,\theta +\varphi u_v)_{0\leq \varphi \leq \pi}$ for $v \in S$. As we will see below, this choice will allow us to identify $H^{0}(M,L^r \otimes \delta)$ with a space spanned by monomials which share the same parity conditions as $r$-admissible colors.
\\
\\ Notice that with these definitions, we have a symplectomorphism between $\mathring{P}\times T \subset \mathbb{R}^E \times T$ and $\mu^{-1}(\mathring{P})\subset \mathcal{M}(\Sigma)$. Furthermore, the complex line bundle $L$ with Hermitian connection $h$ is constructed to have the same curvature and holonomy as the Chern-Simons bundle, and for $\delta$ a bundle on $\mu^{-1}(\mathring{P})$ with the same holonomy called the metaplectic bundle can also be defined, see \cite{Mar09}.   
\\
\\ We now want to build an isomorphism between $V_r(\Sigma)$ and holomorphic sections of $L^r \otimes \delta$ on $M$. The space of holomorphic sections of $L^r \otimes \delta$ has a Hermitian scalar product induced by the Hermitian structures on $L$ and $\delta$:
$$(s,s')=\int_{\mathbb{R}^E \times T} s \overline{s'}e^{-r \varphi}\frac{\omega^{n}}{n!}=\int_{\mathbb{R}^E \times T} s(t,\theta)\overline{s'(t,\theta)}e^{-\frac{r}{2}||t||^2}dt_1\ldots dt_n d\theta_1 \ldots d\theta_n$$
Let $\kappa$ be the constant such that $\frac{1}{\kappa}\int_{\mathbb{R}^E \times T}e^{-\frac{r}{2}||t||^2}dt_1\ldots dt_n d\theta_1 \ldots d\theta_n =1$
\\ (that is, $\kappa=\mathrm{Vol}(T)\left(\frac{r}{2\pi}\right)^{\frac{n}{2}}$, where we recall that $n=|E|$).
 
\begin{proposition}\label{isom}For $\alpha \in I_{\infty}$ the formula $$e_{\alpha}=\frac{z^{\alpha}}{||z^{\alpha}||}=\frac{1}{\sqrt{\kappa}} e^{\frac{t \cdot \alpha}{2}+i \alpha \cdot \theta} e^{-\frac{||\alpha||^2}{4 r}}$$ defines a holomorphic section of $L^r \otimes \delta$, and if we set $H_r=\mathrm{Vect}\{ e_{\alpha} \ , \alpha \in I_r \}$, the map 
$$\begin{array}{rlcl}\Phi_r : & V_r(\Sigma ) & \mapsto & H_r
\\ & \varphi_{\alpha} &\mapsto & e_{\alpha}
\end{array}$$
is an unitary isomorphism between $V_r(\Sigma)$ and $H_r$.
\end{proposition}  
\begin{proof}Indeed the parity conditions for $\alpha \in I_{\infty}$ is exactly what is required for the sections $z^{\alpha}$ to have the correct equivariance to be a section of $L^r \otimes \delta$. It is easy to see then that the $z^{\alpha}$ are orthogonal for the hermitian product on $H^0 (M ,L^r \otimes \delta)$ and form an Hermitian basis of $H^0 (M,L^r \otimes \delta)$. The vector spaces $V_r(\Sigma)$ and $H_r$ have orthonormal basis $\varphi_{\alpha}$ and $e_{\alpha}$ indexed exactly by the same label set $I_r$, hence $\Phi_r$ is indeed an unitary isomorphism.
\end{proof}
We end this section with a technical result that we will need in Section \ref{quasimode}: the asymptotics of the Schwartz kernel of the orthogonal projection $\Pi_r : \mathbb{L}^2(M,L^r \otimes \delta) \rightarrow H_r$. If $s$ is an $\mathbb{L}^2$ section of $L^r \otimes \delta$, then $\Pi_r s$ may be expressed as
$$\Pi_r s (t,\theta) =\underset{\alpha \in I_r}{\sum} \langle e_{\alpha} , s \rangle e_{\alpha}(t,\varphi)=\int_M N(t,\theta,u,\varphi) s(u,\varphi) du d \varphi$$
where $N(t,\theta,u,\varphi)=e^{-\frac{r ||t||^2}{2}}\underset{\alpha \in I_r}{\sum}\overline{e_{\alpha}(u,\varphi)}e_{\alpha}(t,\theta)$ is the Schwartz kernel of $\Pi_r$. There is also a orthogonal projection map: $\Pi_r ' : \mathbb{L}^2(M,L^r \otimes \delta) \rightarrow H^0(M,L^r \otimes \delta)$ whose Schwartz kernel $N'$ is called the Bergman kernel. The Bergman kernel is a section of the bundle $L\boxtimes L^{-1} \otimes \delta\boxtimes \delta^{-1}$ on $M \times M$. The asymptotics of Bergman kernels of a compact Kahler manifold $M$ with prequantizing bundle are well understood and concentrate on the diagonal $\Delta=\lbrace (x,x) \ x \in M \rbrace$, see for example \cite{SZ02}. Here the Kahler manifold we work with is not compact but we have:
\begin{proposition} \label{bergman}The bergman kernel $N(t,\theta,u,\varphi)=e^{-\frac{r ||t||^2}{2}}\underset{\alpha \in I_r}{\sum}\overline{e_{\alpha}(u,\varphi)}e_{\alpha}(t,\theta)$ defined on $M \times M$ is $O(r^N)$ for the norm of the supremum for some $N$, and on any compact subset $K$ of $(\mathring{P}\times T)^2$ we have the asymptotic expansion:
\begin{multline*}N(t,\theta,u,\varphi)=
\\(\frac{r}{2\pi})^{n}2^{6-4g}\exp\left(-\frac{r||t-u||^2}{8}-\frac{r||\theta-\varphi||^2}{2}+ir \left(\frac{t+u}{2}\right)\cdot (\theta-\varphi)\right)e^{i \underset{j}{\sum} (\theta_j-\varphi_j)}+O(r^{-\infty})
\end{multline*}
for the norm of the supremum on $K$. 
\end{proposition}
\begin{proof}
First notice that $$N(t,\theta,u,\varphi)=e^{-\frac{r ||t||^2}{2}}\underset{\alpha \in I_r}{\sum}\overline{e_{\alpha}(u,\varphi)}e_{\alpha}(t,\theta)=\frac{1}{\kappa}\underset{\alpha \in I_r}{\sum}e^{-\frac{r||t-\frac{\alpha}{r}||^2}{4}-\frac{r||u-\frac{\alpha}{r}||^2}{4}}e^{i \alpha \cdot (\theta- \varphi )}$$
is bounded uniformly by a polynomial in $r$ as each term in the sum is bounded by $1$ on $M \times M$ and $\kappa=\mathrm{Vol}(T)\left(\frac{r}{2\pi}\right)^{\frac{n}{2}}$.
\\ Furthermore, elements of $I_r$ are of the form $\alpha=-(1,\ldots,1) +\gamma$ where $\gamma$ is an element of the lattice $\Lambda$ defined in Section \ref{sec:moduli}. On a compact subset $K$ of $(\mathring{P}\times T)^2$, this sum is the same as
$$\frac{1}{\kappa}\underset{\alpha \in \Lambda}{\sum}e^{-\frac{r||t+\frac{1}{r}-\frac{\alpha}{r}||^2}{4}-\frac{r||u+\frac{1}{r}-\frac{\alpha}{r}||^2}{4}}e^{i \alpha \cdot (\theta- \varphi )}=\frac{1}{\kappa}\underset{\alpha \in \Lambda}{\sum} f(r,t,u,\theta-\varphi,\frac{\alpha}{r})$$
where $f(r,t,u,\xi , x)=e^{-\frac{r||t+\frac{1}{r}-x||^2}{4}-\frac{r||u+\frac{1}{r}-x||^2}{4}}e^{irx\xi}$. Note that as a function of the last variable $x$, it is in the Schwartz class $\mathcal{S}(\mathbb{R}^E)$, with all Schwartz semi-norms bounded independently from $r$ and uniform in all other variables. We set $g(x)=f(r,t,u,\xi , x)$ and we use Poisson summation formula: for the Schwartz function $g$ on $\mathbb{R}^E$, we have:
$$\underset{\gamma \in \Lambda}{\sum} g(\frac{\alpha}{r})=\frac{r^n}{\mathrm{Covol}(\Lambda)}\underset{\mu \in \mathrm{Hom}(\mathbb{R}^E / \Lambda ,\mathbb{R}/2 \pi \mathbb{Z})}{\sum} \hat{g}(r \mu)$$
where $\hat{g}$ is the Fourier transform $\hat{g}(\mu)=\int_{\mathbb{R}^E} g(t)e^{-i \mu (t)}dt$. But as $g$ is in the Schwartz class with uniform semi-norms independent of $r$, all terms corresponding to $\mu \neq 0$ are $O(r^{-\infty})$. Finally $\frac{r^n}{\kappa \mathrm{Covol(\Lambda)}}=\left( \frac{r}{2\pi}\right)^{\frac{3 n}{2}}2^{6-4g}$ as $\mathrm{Vol}(T)=(2 \pi)^n \mathrm{Covol}(\Lambda)$ and  $\mathrm{Covol}(\Lambda)=2^{2g-3}$ (see \cite{Mar09}) and thus we get:
\begin{multline*}N(t,\theta,u,\varphi)
=\left(\frac{r}{2\pi}\right)^{\frac{3n}{2}}2^{6-4g}\hat{g}(0)
\\=(\frac{r}{2\pi})^{n}2^{6-4g}\exp\left(-\frac{r||t-u||^2}{8}-\frac{r||\theta-\varphi||^2}{2}+ir \left(\frac{t+u}{2}\right)\cdot (\theta-\varphi)\right)e^{i \underset{j}{\sum} (\theta_j-\varphi_j)}
\end{multline*}
up to $O(r^{-\infty})$ as claimed.
\end{proof} 
\subsection{Curve operators as Toeplitz operators}
The above section explained the construction of a geometric quantization model for $V_r (\Sigma)$, the isomorphism $\Phi_r$ sends $V_r(\Sigma)$ to a space of holomorphic sections on a Kähler manifold $M$ of line bundles $L^r \otimes \delta$. This model is not a very natural object but it has the advantage of being quite simple and explicit, which is what we need for our goal of computing pairings of vectors in $V_r (\Sigma)$.
\\ Having a simple model for $V_r(\Sigma)$, we turn to the curve operators associated to curves on $\Sigma$.
By conjugating a curve operator $T_r^{\gamma} \in \textrm{End}(V_r (\Sigma))$ using the isomorphism $\Phi_r$ between $V_r(\Sigma)$ and $H_r$, the curve operators can be seen as endomorphisms of $H_r$. We wish to understand the curve operators as Toeplitz operators on $H_r$.
\\ Usually given a compact Kähler manifold with a prequantizing bundle $L$ and half form bundle $\delta$, Toeplitz operators are defined as follows: we have an orthogonal projection operator $\Pi_r : \mathbb{L}^2(M,L^r \otimes \delta) \rightarrow H^0(M,L^r \otimes \delta)$ and for $f :M \rightarrow \mathbb{R}$ a smooth function we have an operator $m_f$ acting on sections of $L^r \otimes$ by pointwise multiplication by $f$. A sequence of endomorphisms of $H^0(M,L^r \otimes \delta)$ is then a Toeplitz operator if there exists a function $g(\cdot,r)$ with asymptotic expansion $g(\cdot,r)=g_0 + \frac{1}{r}g_1 +\ldots$ for the norm of the supremum such that:
$$T_r=\Pi_r m_g(\cdot,r)+R_r$$ 
where the term $R_r$ is an operator of norm $O(r_{-N})$ for any $N$. 
\\ We introduce a slightly modified definition of Toeplitz operators as we wish to work with the open manifold $M=\mathbb{R}^E \times T$. Let $\Pi_r$ be the orthogonal projector $\mathbb{L}^2(M,L^r\otimes \delta) \rightarrow H_r$ where $H_r$ is the vector space defined in the last section. Moreover for any smooth bounded function $f$ on $M$, let $m_f$ be the operator on $\mathbb{L}^2 (M,L^r \otimes \delta)$ of multiplication of a section by $f$. 
\begin{definition}Let $U$ be an open subset of $\mathbb{R}^E \times T$ and let $f_0 ,f_1 ,f_2,\ldots$ be a sequence of smooth functions on an open subset $U$. We say that the sequence $T_r$ of endomorphisms of $H_r$ is a Toeplitz operator of symbol $f_0+\frac{1}{r}f_1+\frac{1}{r^2}f_2+\ldots$ on $U$ if for any $k$ and any compact subset $K \subset U$ we have:
\begin{displaymath}T_r =\Pi_r m_{\chi (f_0+\frac{1}{r}f_1+\ldots +\frac{1}{r^k}f_k)} +R_r
\end{displaymath}
where $\chi$ is some smooth function with compact support in $U$ such that $\chi \equiv 1$ on a neighborhood of $K$, and $R_r$ are operators whose norms are $O(r^{-k-1})$ for the norm of operators $(H_r , ||\cdot ||_{\mathbb{L}^2(M)})\rightarrow (H_r ,||\cdot ||_{L^{\infty}(K)})$ for any compact subset $K \subset U$ and are $O(r^N)$ for the norm of operator $H_r \rightarrow H_r$, for some $N \in \mathbb{N}$ 
\end{definition}
Representing curve operators as Toeplitz operators is the main ingredient towards our formula for pairings of curve operator eigenvectors. Indeed, the asymptotic behavior of eigenvectors of Toeplitz operators is well understood, eigenvectors are expected to concentrate on level sets of the principal symbols.
\\ The Theorem \ref{curvop} will serve to identify the principal and subprincipal symbols of curve operators: we will use it to match the asymptotic expansion of matrix coefficients of $T_r^{\gamma}$ with that of a Toeplitz operator of symbol $f=f_0+\frac{1}{r}f_1 +\frac{1}{r^2}f_2+\ldots$. We will find that the appropriate principal symbol is the trace function on $\mathcal{M}(\Sigma)$ associated to the curve $\gamma$.
\\ Our definition of Toeplitz operators is a bit unusual in that Toeplitz operators are usually introduced as having smooth symbol on a compact prequantized manifold $M$. Here we work with on an open manifold $M$ and the symbol we get from Theorem \ref{curvop} might not behave well on the boundary of $P \times T$, hence the need for this local definition of Toeplitz operator. The usual computation of quasimodes of Toeplitz operators derived from microlocal calculus will still work with this definition.
\\
\\ We want to compute the matrix coefficients of some Toeplitz operator of symbol $f$. We will link the matrix coefficients of such a Toeplitz operator to the Fourier coefficient of its symbol.
\\
\begin{lemma}\label{prodscal}
Let $\alpha_r$ be a sequence of admissible colorings such that $\frac{\alpha_r}{r} \in K$ where $K$ is a compact neighborhood of some point $x$ in $\mathring{P}$, and let $e_{\alpha_r}$ be the corresponding basis vectors of $H_r$.
\\
 Let $f$ be a smooth function on $\mathbb{R}^E$ with compact support. Finally, take $k \in \mathbb{Z}^E$ and define $\Delta$ as the differential operator $\underset{e}{\sum} \frac{\partial^2}{(\partial x_e)^2}$. Then there exists differential operators $(L_i)_{i\geq 2}$ on $\mathbb{R}^E$ such that for any $n$ we have: 
\begin{multline*}\Pi_r m_{f(t) e^{i k \theta}} e_{\alpha_r}= 
 \big( f(\frac{\alpha_r}{r})
+\frac{1}{2 r}\big(\Delta f (\frac{\alpha_r}{r}) +k \cdot \bigtriangledown f (\frac{\alpha_r}{r})-\frac{||k||^2}{4} f(\frac{\alpha_r}{r})\big)\\+\underset{n \geq 2}{\sum}\frac{1}{r^n}(L_n f)(x)  +O(r^{-n-1}) \big) e_{\alpha_r+k}
\end{multline*}
 Furthermore, the $O(r^{-n-1})$ are independent of the sequence $\alpha_r$ such that $\frac{\alpha_r}{r} \in K$.
\end{lemma}
\begin{proof}It is straightforward from the definition of $e_{\alpha_r}$ that $f(t)e^{i k \theta} e_{\alpha_r}$ is orthogonal to $e_{\alpha_r+l}$ for any $l \neq k$. Thus $\Pi_r m_{f(t)e^{i k \theta}}e_{\alpha_r}$ is colinear to $e_{\alpha_r+k}$ and we only need to estimate the coefficient
$$\begin{array}{rl}(f(t)e^{i k \theta}e_{\alpha_r} ,e_{\alpha_r+k}) & =\frac{1}{\kappa}\int_{M}f(t)e^{-\frac{r}{2}||t||^2}e^{\alpha_r \cdot t + \frac{k}{2}\cdot t-\frac{||\alpha_r||^2}{2 r}-\frac{k \cdot \alpha_r}{2 r}-\frac{||k||^2}{4 r}}
\\ &=\frac{e^{-\frac{k \cdot \alpha_r}{2 r}-\frac{||k||^2}{4 r}}}{\kappa}\int_{M} g(t) e^{-\frac{r}{2}||t-\frac{\alpha_r}{r}||^2} 
\end{array}$$
where we set $g(t)=f(t)e^{\frac{k}{2}\cdot t}$. 
\\ A stationary phase lemma argument will give us an asymptotic expansion of the integral. Indeed, let $K'$ be a compact neighborhood of $K$, as $\frac{\alpha_r}{r}\in K$ for all $r$, for $x \in M \setminus K'$ we have $|g(t)|e^{-\frac{r}{2}||t-\frac{\alpha_r}{r}||^2} \leq ||g||_{\infty} e^{-r d(x,K)^2}$. As $d(K, M\setminus K')>0$, the integral of $g(t) e^{-\frac{r}{2}||t-\frac{\alpha_r}{r}||^2}$ on $M \setminus K'$ is a $O(r^{-k})$ for every $k$, with constants independent of $\frac{\alpha_r}{r}$.
\\ Furthermore, we write the Taylor expansion at $\frac{\alpha_r}{r}$ of $g$ on $K'$ at order $k$:
\begin{displaymath}g(t)=g(\frac{\alpha_r}{r})+D g(\frac{\alpha_r}{r})(t-\frac{\alpha_r}{r})+\ldots +\frac{1}{k!}D^{k}g(\frac{\alpha_r}{r}) (t-\frac{\alpha_r}{r})+h(t)
\end{displaymath} 
with $|h(t)|\leq \frac{C_k}{(k+1)!}(t-\frac{\alpha_r}{r})^{k+1}$ where $C_k$ is the supremum of $||D^{k+1}g||$ on $K'$, which is an universal constant independent of $\frac{\alpha_r}{r}$. 
\\ Integrating by part each integral $\int_M D^{k}g(\frac{\alpha_r}{r})(t-\frac{\alpha_r}{r})e^{-\frac{r}{2}||t-\frac{\alpha_r}{r}||^2}dt$, we get $0$ whenever $k$ is odd, $\frac{1}{r}\Delta g (\frac{\alpha_r}{r})$ if $k=2$ and $\frac{1}{r^{n}}L_{n}(g)(\frac{\alpha_r}{r})$ when $k=2n$ and where $(L_n g) (\frac{\alpha_r}{r})$ is a linear combination of the degree $k$ derivatives of $g$. We get the asymptotic expansion: 
\begin{equation*}\frac{1}{\kappa}\int_{M} g(t) e^{-\frac{r}{2}||t-\frac{\alpha_r}{r}||^2}
=g(\frac{\alpha_r}{r})
+\frac{1}{2 r}\Delta g (\frac{\alpha_r}{r}) +\sum \frac{1}{r^k}L_k g(\frac{\alpha_r}{r}) +O(r^{-k-1})
\end{equation*}
where $\Delta g=\underset{e \in E}{\sum}\frac{\partial^2}{(\partial x_e)^2}$ and $L_k$ are some differential operators of degree $2 k$. The $O(r^{-k-1})$ is uniform for $\frac{\alpha_r}{r} \in K$.
\\ As $g(t)=f(t)e^{\frac{k}{2}t}$, the derivatives of $g$ can be computed in terms of the derivatives of $f$. 
We find that there are differential operators $L_n '$ of degree less than $2 n$ such that $(L_n g) (t)=(L_n ' f )(t) e^{\frac{k}{2}t}$ and we have $\Delta g (t)=\left( \Delta f (t) +\frac{1}{2}k \cdot \nabla f (t) +\frac{||k||2}{4}f(t)\right) e^{\frac{k}{2}t}$ where $\nabla$ is the gradient, $\cdot$ is the scalar product in $\mathbb{R}^E$ and $||\cdot ||$ is the Euclidian norm in $\mathbb{R}^E$.
\\ Hence the matrix coefficient $(f(t)e^{i k \theta}e_{\alpha_r} ,e_{\alpha_r+k})$ has the asymptotic expansion:
\begin{multline*}(f(t)e^{i k \theta}e_{\alpha_r} ,e_{\alpha_r+k})= f(\frac{\alpha_r}{r})
+\frac{1}{2 r}\left(\Delta f (\frac{\alpha_r}{r}) +k \cdot \bigtriangledown f (\frac{\alpha_r}{r})-\frac{||k||^2}{4} f(\frac{\alpha_r}{r})\right)\\+\underset{n \geq 2}{\sum}\frac{1}{r^n}(L'_n f)(x)  +O(r^{-n-1})
\end{multline*}
where $L'_n$ are some differential operators of degree less than $2 n$.
\\ The error factor $O(r^{-n-1})$ is again uniform for $\frac{\alpha_r}{r} \in K$.
\end{proof}

\begin{theo}
\label{toeplitz}For any simple closed curve $\gamma$ on $\Sigma$, there exists functions $f_k \in C^{\infty}(\mathring{P}\times T)$ such that $T_r^{\gamma} \in \textrm{End}(H_r)$ is a Toeplitz operator on $\mathring{P}\times T$ of symbol $f_0+\frac{1}{r}f_1+\ldots$
\\
Furthermore the principal symbol $f_0$ of $T_r^{\gamma}$ is the trace function $\sigma^{\gamma}(t,\theta)=-\tr (\rho_{t,\theta}(\gamma))$ and the Weyl subprincipal symbol defined as $f_1+\frac{1}{2}\Delta_{\partial} f_0$ vanishes, where $\Delta_{\partial}$ is the Kähler Laplacian on $M$.

\end{theo}
\begin{proof}
First we want to introduce the functions $f_0,f_1 ,\ldots $ that constitutes the symbol of $T_r^{\gamma}$, using Theorem \ref{curvop} and Lemma \ref{prodscal}. 
Let $K$ be a compact in $\mathring{P}\times T$ and let $K'$ be a compact neighborhood of $K$ in $\mathring{P}\times T$. We choose a function $\chi$ with compact support in $\mathring{P}\times T$ which is identically $1$ on $K'$.
\\ Let also $\alpha_r$ be a sequence such that $\frac{\alpha_r}{r} \in K'$.
\\
According to Theorem \ref{curvop}, to represent $T_r^{\gamma}$ as a Toeplitz operator $T_r^{f_r}$ acting on $H_r$ with $f_r=f_0+\frac{1}{r}f_1+\frac{1}{r^2}f_2+\ldots$, and $f_i^j$ is the $j$-th Fourier coefficient of $f_i$, we need to have:
\begin{multline*}F_{k}^{\gamma}(\frac{\alpha}{r},\frac{1}{r})=\langle T_r^{\gamma}\varphi_{\alpha} , \varphi_{\alpha+k}\rangle =(T_r^{f_r}e_{\alpha},e_{\alpha+k})
\\=f_0^{k}(\frac{\alpha}{r})+\frac{1}{r}\left(f_1^{k}(\frac{\alpha}{r})+\frac{1}{2}\left[\Delta f_0^k (\frac{\alpha}{r}) +k \cdot \bigtriangledown f_0^k (\frac{\alpha}{r})-\frac{||k||^2}{4} f_0^k(\frac{\alpha}{r})\right]\right)+O(r^{-2})
\end{multline*}
using Lemma \ref{prodscal} for the second equality.
\\
Gathering the equations for each Fourier coefficient and using Theorem \ref{curvop}, we get:
$$f_0(t,\theta)=\underset{k:E \rightarrow \mathbb{Z}}{\sum}F_k^{\gamma}(t,0)e^{i k \theta}=-\textrm{Tr}(\rho_{t,\theta}(\gamma))=\sigma^{\gamma}(t ,\theta)$$
and 
$$f_1(t,\theta)+\underset{e \in E}{\sum}\left[\frac{1}{2}\frac{\partial^2}{\partial t_e^2}+\frac{1}{2 i}\frac{\partial^2}{\partial t_e \partial \theta_e}+\frac{1}{8}\frac{\partial^2}{\partial \theta_e^2}\right]f_0 (t,\theta)=\frac{1}{2 i}\underset{e \in E}{\sum}\frac{\partial^2}{\partial t_e \partial \theta_e}\sigma^{\gamma}$$
Remember that $w_e=\frac{t_e}{2}+i \theta_e$ are local complex coordinates such that $\omega=i \sum d w_e \wedge d\overline{w_e}$, thus the Kähler laplacian $\Delta_{\partial}$ on $M$ is simply $\sum \frac{\partial^2}{\partial w_e \partial \overline{w_e}}=\sum \frac{\partial^2}{\partial t_e^2}+\frac{1}{4}\frac{\partial^2}{\partial \theta_e^2}$. 
\\ Thus $f_1+\frac{1}{2}\Delta_{\partial}f_0$ must vanish.
\\ It is then possible to choose further coefficients $f_k$ to match the asymptotic expansion up to $O(r^{-k+1})$ for each $k$, simply choosing $f_{k+1}$ to cancel the residual term in $\frac{1}{r^{k+1}}$ in $T_r^{\gamma}-\Pi_r m_{f_0 +\frac{1}{r}f_1 +\ldots \frac{f_k}{r^k}}$.
\\ 
\\ At this state, we have introduce smooth functions $f_0,f_1,f_2,\ldots$ such that for any $K$ compact subset of $\mathring{P}\times T$, any $K'$ compact neighborhood of $K$ and $\chi \equiv 1$ on $K'$ of compact support, we have 
$$\left(T_r^{\gamma}-\Pi_r m_{\chi (f_0+\frac{1}{r}f_1+\ldots \frac{1}{r^k}f_k)}\right) s_r =O(r^{-k-1})_{L^{\infty}(K)}$$
uniformly for $s_r \in \mathrm{Vect}(e_{\alpha_r}\in H_r \ / \ \frac{\alpha_r}{r}\in K' )$ of norm $1$. 
\\ So we just have to control the difference of the two operators on the subspace 
\\ $\textrm{Vect}(e_{\alpha} \in H_r, / \, \frac{\alpha}{r} \in P-K')$. But there is a constant $C$ such that $\forall \alpha \in P-K'$, we have $\underset{K}{\textrm{sup}}(|e_{\alpha}|^2)\leq C e^{-r d(K,P-K')^2}$. Thus for $s_r \in \textrm{Vect}(e_{\alpha} \in H_r, / \, \frac{\alpha}{r} \in P-K')$ of norm $1$, we have $\underset{K}{\textrm{sup}}|s_r|<C r^{-k-1}$ for an constant $C$ not depending on $s_r$. As the operators  $T_r^{\gamma}$ sends $e_{\alpha}$ to a linear combination of $e_{\alpha+k}$ with $k$ bounded, and as the $T_r^{\gamma}$ are bounded for the norm of operators on $H_r$, we must also have $\underset{K}{\textrm{sup}}|T_r^{\gamma} s_r|<C' r^{-k-1}$.
\end{proof}
 The quantity $f_1+\frac{1}{2}\Delta_{\partial} f_0$ is sometimes called in the litterature the Weyl-subprincipal symbol of the Toeplitz operator $T_r^{f}$. A straightforward computation gives that, for $T_r^{f}$ and $T_r^{g}$ two Toeplitz operators of principal symbols $f_0$ and $g_0$ and subprincipal symbols $f_1$ and $g_1$, the composition $T_r^{f}T_r^{g}$ has $f_0 g_0$ as principal symbol and $f_0 g_1 +f_1 g_0 +\frac{1}{2}\lbrace f ,g \rbrace$ as subprincipal symbol, which is the composition law that a Weyl-subprincipal symbol ought to satisfy.
\\
\\
With the framework we developed in the last paragraphs, we wish to study the following problem: take $\Sigma$ a closed oriented surface of genus $g$ and $\mathcal{C}=(C_e)_{e \in E}$ a pants decomposition of $\Sigma$. Such a pants decomposition gives rise to a moment application $\mu :\mathcal{M}(\Sigma) \rightarrow P$, to basis $\varphi_{\alpha}$ of $V_r(\Sigma)$ where $\alpha$ are some integer points of the moment polytope $P$ and to isomorphisms $\phi_r : V_r(\Sigma) \longrightarrow H_r$ defined above, where $H_r$ are subspaces of $H^0(M,L^r\otimes \delta)$. Suppose $\mathcal{D}=(D_f)_{f \in F}$ is another pants decomposition, and the associated basis of $V_r(\Sigma)$ is $\psi_{\beta}$. As $V_r(\Sigma)$ has an Hermitian product, we can form the pairings $\langle \varphi_{\alpha} ,\psi_{\beta}\rangle$ and study the limit as $\frac{\alpha}{r}$ and $\frac{\beta}{r}$ tend to some limits in $\mathbb{R}^E$ and $\mathbb{R}^F$.
\\ But the vectors $\psi_{\beta}$ are joint eigenvectors of the curve operators $T_r^{\gamma_i}$, and by Theorem \ref{toeplitz}, we know that the operators $T_r^{\gamma_i}$ act as Toeplitz operators with symbol $\sigma^{\gamma_i}$ on $H^0(\mathring{P}\times T,L^r \otimes \delta)$. Eigenvectors of Toeplitz operators are well understood, in particular they concentrate on level sets of the principal symbols. We will be able to give an asymptotic form for $\psi_{\beta}$ as a section of $L^r\otimes \delta$, and thus we will be able to compute the pairing with $\varphi_{\alpha}$. As $\varphi_{\alpha}$ and $\psi_{\beta}$ concentrate on level sets of $\sigma^{C_e}$ and $\sigma^{D_f}$, the formula will be a sum of contributions coming from each intersection points of these level sets. The next section is to devoted to a result proving that generically, such level sets intersect nicely in a finite number of points, allowing us to obtain an asymptotic formula by summing the contributions of each of these points to the pairing in Section \ref{sec:pairing}.
\\ Such pairings have an interpretation as quantum invariants: $\langle \varphi_{\alpha},\psi_{\beta} \rangle$ is the Reshetikhin-Turaev invariant of the 3-manifold with links obtained by gluing the handlebodies associated to the pants decompositions $\mathcal{C}$ and $\mathcal{D}$, and adding in each handlebody the dual graph of the pants decomposition colored by $\alpha$ and $\beta$ respectively. A special case would be if both coloring are trivial colorings $(1,\ldots,1)$, then adding the colored trivalent graphs correspond to adding empty links in each handlebody, thus we would obtain the Reshetikhin-Turaev invariants of the 3-manifold with Heegard genus $g$ and Heegard splitting corresponding by the pants decomposition $\mathcal{C}$ and $\mathcal{D}$. Unfortunately our approach to calculating pairings fails in this case, as some intersection points will be in $\partial P \times T$ and we lack control over what happens on $\partial P$.

\section{ Intersections of Lagrangians in $\mathcal{M}(\Sigma)$}
\label{sec:lag}
We fix a pants decomposition $\mathcal{C}=(C_e)_{e \in E}$, which defines an isomorphism $\Phi_r:V_r(\Sigma)\rightarrow H_r$ as in Proposition \ref{isom}. Furthermore, we denote by  $U$ the set $\mu^{-1}(\mathring{P}) \subset \mathcal{M}(\Sigma)$, where $\mu$ is the moment map defined as in Subsection \ref{sec:moduli}.
\\ For $\mathcal{D}=(\gamma_i)_{i \in I}$ a pants decomposition of $\Sigma$, we introduce the closed subset $\Lambda_x^{\mathcal{D}}$ of $\mathcal{M}(\Sigma)$ defined by 
\begin{displaymath}
\Lambda_x^{\mathcal{D}}=\lbrace \rho ,\, \forall i \in I , \, -\tr (\rho(\gamma_i))=x_i \rbrace
\end{displaymath}
 When $x$ is in the interior of the moment polytope associated to the pants decomposition $\mathcal{D}$, these subsets are Lagrangian tori of $\mathcal{M}(\Sigma)$: indeed, it is  the pre-image of a regular value of the Poisson commuting trace functions $\tr(\rho(\gamma_i))$. The Arnold-Liouville theorem ensures it is a torus of dimension $n$ where $\textrm{dim}(\mathcal{M}(\Sigma))=2 n$. 
\\
\\ As we expect the joint eigenvectors of curve operators $T_r^{\gamma_i}$, viewed as elements of $H_r$ to concentrate on such Lagrangians, we wish to show that they have nice properties for generic $x$. 
\begin{proposition}\label{laginter}For $\mathcal{D}=(\gamma_i)_{i \in I}$ and $\mathcal{F}=(\delta_j)_{j\in J}$ any pants decompositions of $\Sigma$, we have:
\begin{itemize}
\item[-]For any $x$ in an open dense subset of $\mathbb{R}^I$, the intersection $\Lambda_x^{\mathcal{D}}\cap \mu^{-1}(\partial P)$ is transverse and $\Lambda_x^{\mathcal{D}}\setminus \mu^{-1}(\partial P)$ is connected. 
\item[-]For any $x,y$ in an open dense subset of $\mathbb{R}^I \times \mathbb{R}^J$, the intersection $\Lambda_x^{\mathcal{D}}\cap \Lambda_y^{\mathcal{F}}$ is transverse.
\end{itemize}
\end{proposition}
\begin{proof}The proposition follows from two steps. First we can obtain the transversality conditions as an application from a classical result in real algebraic geometry, the algebraic Sard theorem. We will shortly introduce the notions needed to state this result, a detailled background is found in \cite{BPR}. 
\\ 
\\ To begin with, we define a semi-algebraic set $N$ as a subset of some $\mathbb{R}^N$ defined by polynomial equations or inequations (strict or large): there are a families of polynomials $P_1, \ldots ,P_n$, $Q_{1}, \ldots Q_{m}$, and $R_1 ,\ldots R_l$ such that \begin{displaymath}N=\lbrace x \in \mathbb{R}^N \,  / \, P_1(x)=0 ,\ldots P_n (x)=0 , Q_1(x)>0 ,\ldots Q_m(x)>0 ,R_1(x)\geq 0 ,\ldots R_l \geq 0\rbrace
\end{displaymath}
Any affine algebraic variety is a semi-algebraic set (defined by equations only). For the usual topology on $\mathbb{R}^N$, an affine algebraic variety is a stratified manifold. Each of its strata are then semi-algebraic sets. 
\\Moreover any semi-algebraic set is also a stratified manifold. The dimension of a semi-algebraic set is then defined to be the maximal dimension of any of its strata.
\\ Regular maps between semi-algebraic sets still are those given by polynomial functions. Finally, semi-algebraic sets have tangent spaces defined in the same manner as in the case of algebraic varieties. 
\\
\\We can now express the algebraic Sard theorem:
\begin{theo}\cite{BPR}Let $N$ and $M$ be two semi-algebraic sets. If $f :N \rightarrow M$ is an algebraic map, then the set $\mathrm{Crit}(f)=\lbrace f(x) \, / \dim Tf_x (T_x N)<\dim(T_{f(x)} M) \rbrace$ is semi-algebraic and has dimension $< \dim M$
\end{theo}
 Now we note that the intersections occurring in Proposition \ref{laginter} are intersections of real algebraic subvarieties inside $\mathcal{M}(\Sigma)$.
\begin{lemma} $Y=\mu^{-1}(\partial P)$ is a real subvariety of $\mathcal{M}(\Sigma)$. Moreover for any $\mathcal{D}$ pants decomposition of $\Sigma$ and any $x \in \mathbb{R}^I$, the set $\Lambda_x^{\mathcal{D}}$ is a real algebraic subvariety of $\mathcal{M}(\Sigma)$.
\end{lemma}
\begin{proof} Indeed, given three curves $C_e,C_f,C_g \in \mathcal{C}$ that bound a pair of pants $Q$, the coordinates $x_e,x_f$ and $x_g$ of $\mu (\rho)$ satisfy three triangular identities of the type $x_e \leq x_f+x_g$, and the inequation $x_e+x_f+x_g \leq 2$. We have equality in one of these equations if and only if the restriction of $\rho$ to the pair of pants $Q$ is commutative. Hence, $\rho \in \mu^{-1}(\partial P)$ if and only if its restriction to one of the pants is commutative, and the set $Y$ is the reunion of the subvarieties $\lbrace \tr (\rho ([C_e ,C_f])=2 \rbrace$, for $C_e$ and $C_f$ in the same pair of pants. 
\\ The case of $\Lambda_x^{\mathcal{D}}$ is straightforward as $\Lambda_x^{\mathcal{D}}$ is defined as the set $\lbrace \rho \ / \ \mathrm{Tr}(\rho(D_i))=x_i \rbrace$.\end{proof}
Topologically, as $Y$ is a real algebraic manifold, it is a stratified manifold. Its strata are in turn semi-algebraic sets.
\\ 
We can apply the algebraic Sard theorem to the map $(f_{D_e})_{e \in E} : \mathcal{M}(\Sigma) \rightarrow \mathbb{R}^I$ restricted to any stratum $Z \subset Y$. We obtain that for $x$ in a dense open subset of $\mathbb{R}^I$, the map $(f_{D_e})_{e \in E}$  does not have $x$ as a critical value on the stratum $Z$. This is the same as saying that $\Lambda_x^{\mathcal{D}}$ is transverse to the stratum $Z$ of $Y$, so for generic $x$ it is transverse to each stratum of $Y$. The same applies to showing the transversality of $\Lambda_x^{\mathcal{D}}$ and $\Lambda_y^{\mathcal{F}}$ for generic $x$ and $y$.
\\
\\ The only thing that remains to prove is the part about the connectedness of $\Lambda_x^{\mathcal{D}}\setminus Y$. We will need the following lemma:
\begin{lemma}\label{codim}The real algebraic subvariety $Y=\mu^{-1}(\partial P)$ has codimension $2$ in $\mathcal{M}(\Sigma)$.
\end{lemma}
\begin{proof} Recall that the subvariety $Y$ is included in the union of subvarieties $$\lbrace \rho \ / \mathrm{Tr}(\rho ([C_e,C_f]))=2 \rbrace$$ where $C_e$ and $C_f$ are curves of the pair of pants decomoposition that bound a common pair of pants. Thus we only have to show that subvarieties of this type have codimension at least $2$.
\\  Let $\gamma$ and $\delta$ two disjoint non-isotopic simple closed curves in the surface $\Sigma$. We show that the subvariety $\lbrace \tr (\rho ([\gamma ,\delta]))=2 \rbrace$ has codimension at least $2$ in $\mathcal{M}(\Sigma)$. This contains the subvariety of abelian representations, which has dimension $2 g$ and thus codimension greater than $2$, as $g \leq 2$. Therefore we are interested in the codimension near an irreducible representation $\rho$.
\\ One possibly is that $\rho \in \lbrace \rho \in \mathcal{M}_{irr}(\Sigma) \, / \, \rho(\gamma)=\pm I \rbrace$. This semi-algebraic set is of codimension at least $3$ in $\mathcal{M}(\Sigma)$. 
\\ Indeed, consider the projection $\pi : \mathcal{M}(\Sigma) \rightarrow \textrm{Hom}(\pi_1 \Sigma , \textrm{SO}_3)/\textrm{SO}_3$. This map is a cover on its image, and thus conserves dimension. The representation $\rho$ is sent to a representation $\tilde{\rho}$ such that $\tilde{\rho}(\gamma)=I$ in $\textrm{SO}_3$. Taking $\tilde{\rho}(\gamma)=I$ amounts to replace $\mathcal{M}(\Sigma)$ with the moduli space of $\Sigma //\gamma$, that is we smash $\gamma$ to a point.
\\ We have two cases: either $\gamma$ is separating, we obtain the wedge of two surfaces of genus $g_1$ and $g_2$ with $g=g_1+g_2$, whose fundamental space is $\pi_1 \Sigma_{g_1} * \pi_1 \Sigma_{g_2}$, and whose moduli space has dimension $6 g_1-6 +6 g_2 -6=6 g-12$. When the curve $\gamma$ is non-separating, $\Sigma // \gamma$ has fundamental group $\pi_1 \Sigma_{g-1} * \mathbb{Z}$ and the moduli space has dimension $6 (g-1)-6+3=6 g-9$. In either case, the codimension is greater than $3$.
\\
\\ Thus we need only to show that $\lbrace \tr (\rho ([\gamma ,\delta]))=2 \rbrace$ has codimension at least $2$ in the neighborhood of points $\rho$ such that $\rho(\gamma) \neq \pm I$ and $\rho(\delta)\neq \pm I$.
\\ We denote by $F$ the function $\rho \rightarrow \tr (\rho ([\gamma ,\delta]))$. Let $\rho$ an irreducible representation in $\lbrace F(\rho)=2 \rbrace$ with $\rho(\gamma) \neq \pm I$ and $\rho (\delta) \neq \pm I$. As we consider representation in $\su$, the function $F$ has a local maximum at $\rho$, thus the differential $D_{\rho}F$ vanishes. To understand the local structure of $\lbrace F=2 \rbrace$ near $\rho$, we compute the 2nd differential of $F$. We will exhibit a subspace of the tangent space of dimension $2$ on which $D^2 F$ is definite negative. This will prove that the tangent space of $\lbrace \tr (\rho ([\gamma ,\delta]))=2 \rbrace$ has codimension at least $2$, and hence finish the proof of our claim.
\\ \textbf{Claim}: There is a pair of pants $P'$ such that $\rho$ is commutative on $P'$ and the restriction map $\mathcal{M}(\Sigma) \rightarrow \mathcal{M}(P')$ is a submersion.
\\ Indeed it is an elementary fact that the tangent space $T_{\rho} \mathcal{M}(\Sigma)$ is isomorphic to the twisted cohomology group $H^1 (\Sigma , \ad \rho)$ where $\textrm{Ad} \rho $ stands for the adjoint representation of $\rho$ (see for example \cite{Mar09}). 
\\ Consider the exact sequence in twisted cohomology associated to the pair $(\Sigma ,P)$:
\begin{displaymath} H^1( \Sigma ,\ad \rho)\rightarrow H^1 (P ,\ad \rho) \rightarrow H^2 (\Sigma , P ,\ad \rho)
\end{displaymath}
By Poincaré duality, we have $H^2(\Sigma , P , \ad \rho) \simeq H^0 (\Sigma \setminus P , \ad \rho)^*$. If $\Sigma \setminus P$ is connected, $\rho$ must be irreducible on $\Sigma \setminus P$ and then $H^0(\Sigma \setminus P ,\ad \rho) =0$. When it is not the case, either $\rho$ is irreducible on each components of $\Sigma \setminus P$ and $H^0(\Sigma \setminus P ,\ad \rho) =0$, or there is another pair of pants $P'$ in the decomposition on which $\rho$ is commutative and such that $\Sigma \setminus P'$ is connected: just take one of the connected components of $\Sigma \setminus P$ on which $\rho$ is commutative, it is a surface with one or two boundary curves, and any decomposition of such a surface has a nonseparating pair of pants disjoint from the boundary. 
\\ Hence, we can always assume that $H^0(\Sigma \setminus P ,\ad \rho) =0$ and thus that the restriction map is a submersion. 
\\ 
\\ Now we only have to show that $\lbrace \rho' \, / \, F(\rho')=2 \rbrace$ is of codimension $2$ in $\mathcal{M}(P)$. 
\\ We now compute the second derivative of the restriction $F|_{\mathcal{M}(P)}$. As $\rho(\gamma)$ and $\rho(\delta)$ commute, up to conjugation we can assume that $\rho(\gamma)$ and $\rho(\delta)$ are diagonal with coefficients $(e^{i \theta},e^{-i \theta})$ and $(e^{i \varphi},e^{-i \varphi})$ respectively. We denote these diagonal matrices by $U_{\theta}$ and $U_{\varphi}$ respectively. We compute the second differential of $F|_{\mathcal{M}(P)}$ on $H^1 (P ,\ad \rho)$, space which is isomorphic to 
\begin{displaymath}H^1 (P ,\ad \rho )=\mathrm{su}_2 \bigoplus \mathrm{su}_2 / \lbrace ( \xi - U_{\theta} \xi U_{\theta}^{-1} ,\xi - U_{\varphi} \xi U_{\varphi}^{-1} , \, \, \xi \in \mathrm{su}_2 \rbrace
\end{displaymath}
Let us introduce the notations: $j=\begin{pmatrix}0 & 1 \\ -1 & 0\end{pmatrix}$ and $k=\begin{pmatrix}0 & i \\ i & 0\end{pmatrix}$. 
\\Now the vector space $V=\lbrace (\xi ,0) ,\xi \in \textrm{Vect}(j ,k) \rbrace$ is a subspace of dimension $2$ of $H^1 (P ,\ad \rho)$ (as $U_{\theta}$ and $U_{\varphi}$ have the same commutant, no $(\xi ,0)$ is equivalent to $(0,0)$ in $H^1(P ,\ad \rho)$). We can endow it with a norm $||\cdot||$ for which $(j,k)$ is an orthonormal basis. 
\\ We show that $D^2 F$ is definite negative on $V$: we have 
\begin{multline*}\tr (U_{\theta} e^{\xi}U_{\varphi}e^{-\xi} U_{-\theta -\varphi})=2-\tr(U_{\theta}\xi U_{\varphi} \xi U_{-\theta-\varphi} )+2 \tr (\xi^2) +O(||\xi||^3)
\\=2+\tr(\xi^2 U_{-2 \varphi})-2 ||\xi||^2 +O(||\xi||^3)
\\ =2+2||\xi||^2(\cos (2\varphi )-1)+O(||\xi||^3)=2 -4 \sin (\varphi) ||\xi||^2+O(||\xi||^3)
\end{multline*}
As the second differential is definite negative on a subspace of dimension $2$, in a neighborhood of $\rho$ the space $\lbrace \rho' \, / \, F(\rho')=2 \rbrace$ is of codimension at least $2$. 
\end{proof}
To finish the proof of Proposition \ref{laginter} we deduce the connectedness of $\Lambda_x^{\mathcal{D}} \cap Y$ for generic $x$ from the Lemma \ref{codim}. Recall that for generic $x$ the intersection  $\Lambda_x^{\mathcal{D}}\cap Y$ is transverse, thus the intersections of $\Lambda_x^{\mathcal{D}}$ with each $\lbrace \tr (\rho ([C_e,C_f]))=2 \rbrace$ are transverse. Also, for generic $x \in \mathbb{R}^n$ we can assume that the set $\Lambda_x^{\mathcal{D}}$ is either empty or a torus of dimension $n$ inside $\mathcal{M}(\Sigma)$. In the former case the connectedness is trivial. In the latter case, to show that $\Lambda_x^{\mathcal{D}}\cap \mu^{-1}(\mathring{P})$ is connected, we will only need to show that $\Lambda_x^{\mathcal{D}}\cap Y$ is of codimension at least $2$ in $\Lambda_x^{\mathcal{D}}$. What we mean by codimension at least $2$, is that each strata of this stratified variety has topological codimension at least $2$ in the torus $\Lambda_x^{\mathcal{D}}$. As the intersection $\Lambda_x^{\mathcal{D}}\cap Y$ is transverse, it follows from the fact that $Y$ is of codimension at least $2$ in $\mathcal{M}(\Sigma)$.
\end{proof}
\section{Pairings of eigenvectors of curve operators}
\label{sec:pairing}
\subsection{Pairing in the half-form bundle}
\label{sec:halfform}
In this short preliminary section, we define various pairings for the half-form bundle $\delta$ on Kähler manifold $M$. These pairing forms will be useful to describe the asymptotic expansions occuring in the pairing of quasimodes in Section \ref{sec:formula}. 
\\We consider a general Kähler vector space $E$ of complex dimension $n$ with symplectic form $\omega$ and complex structure $J$. 
\\ Choose two transverse Lagrangian subspaces $\Gamma_1$ and $\Gamma_2$ of the vector space $E$.
\\ Let $\Lambda^{n,0}E^*$ be the space of complex $n$-forms on $M$, of which $\delta$ is a square root. We have maps:
$$\pi_i: \Lambda^{n,0}E^* \rightarrow \Lambda^n \Gamma_i^* \otimes \mathbb{C}$$ 
which consist of restricting a complex $n$-form on $E$ to $\Gamma_i$, getting an isomorphism between complex $n$-forms on $E$ and the complexification of real $n$-forms on $\Gamma_i$. On the other hand we have maps
$$\Lambda^n \Gamma_i^* \otimes \mathbb{C} \rightarrow \Lambda^{n,0} E^*$$
extending $n$-form on $\Gamma_i$ to complex $n$-form on $E$. These maps are well defined as the $\Gamma_i$ are Lagrangian (thus $\Gamma_i \bigoplus J \Gamma_i =E$), and are the inverse isomorphisms of the first couple of maps. 
\\ Now given $n$-forms on $\Gamma_1$ and $\Gamma_2$ the wedge product create an $2 n$-form on $E$, which we can compare with the Liouville form $\frac{\omega^n}{n!}$. Combining with the restriction maps $\pi_i$, we get the pairing 
$$\begin{array}{ccccc} \Lambda^{n,0}E^* & \times & \Lambda^{n,0}E^* & \rightarrow & \mathbb{C}
\\ \alpha & , & \beta & \rightarrow & (\alpha ,\beta)_{\Gamma_1,\Gamma_2} =i^{n(2-n)}\frac{\pi_1 (\alpha) \wedge \pi_2 (\beta)}{\omega^n}
\end{array}$$
If we have a complex line $\delta$ with an isomorphism $\phi :\delta^{\otimes 2} \rightarrow \Lambda^{n ,0 }E^* =\mathbb{C}$, a pairing for $\delta$ associated to Lagrangians $\Gamma_1$ and $\Gamma_2$ is 
$$(\alpha ,\beta)_{\Gamma_1 ,\Gamma_2}=\sqrt{(\alpha^{\otimes 2} ,\beta^{\otimes 2})_{\Gamma_1 ,\Gamma_2}}$$
The determination of the square root goes as follows: recall that we also have an Hermitian pairing
$$\begin{array}{ccccc} \Lambda^{n,0}E^* & \times & \Lambda^{n,0}E^* & \rightarrow & \mathbb{C}
\\ \alpha & , & \beta & \rightarrow & (\alpha ,\beta) =i^{n(2-n)}\frac{\alpha \wedge \overline{\beta}}{\omega^n}
\end{array}$$
and thus also a pairing $(\alpha, \beta)= \sqrt{(\alpha^{\otimes 2} ,\beta^{\otimes 2})}$
\\ It is shown in \cite{Cha10a} that when $\Gamma_2=J \Gamma_1$, this pairing is the same as the pairing $(\alpha ,\beta)_{\Gamma_1 ,\Gamma_2}$ up to a positive constant. We require the same to be true for the corresponding pairing on $\delta$, for general transverse $\Gamma_1$ and $\Gamma_2$ we extend the definition so that it depends continuously on $\Gamma_1$ and $\Gamma_2$.
\\
\\ Suppose we consider, instead of just a Kähler vector space, a (connected) Kähler manifold $M$ of complex dimension $n$, equipped with a half-form bundle $\delta$. Pick a point $x$ in $M$, pairings on $\delta_x$ can be defined by the above procedure. We get pairings on any $\delta_y$ for $y \in M$ by extending these by continuity. 
\subsection{Quasimodes of curve operators}
\label{sec:quasimodes}
Let $\mathcal{C}=(C_e)_{e \in E}$ and $\mathcal{D}=(D_f)_{f \in F}$ be two collections of curves that form pair of pants decompositions of $\Sigma$. We assume that these pair of pants decompositions have planar dual graphs. We call $I_r$ (resp. $J_r$) the set of $r$-admissible colorings associated to the pair of pants decomposition $\mathcal{C}$ (resp. $\mathcal{D}$), $(\varphi_{\alpha}^r)_{\alpha \in I_r}$ and $(\psi_{\beta}^r)_{\beta \in J_r}$ the associated basis of $V_r(\Sigma)$.
\\ 
\\ We introduce the polytope $P$ (resp. $Q$) that is the image of $\mathcal{M}(\Sigma)$ by the momentum mapping $\mu: \mathcal{M}(\Sigma) \rightarrow \mathbb{R}^E$ such that $\mu(\rho)=\textrm{acos}(\frac{1}{2}\tr (\rho(C_e)))$ (resp. $\mu '$ such that $\mu ' (\rho)=\textrm{acos}(\frac{1}{2}\tr (\rho(C_e)))$.
\\ We want to study the asymptotic behavior of pairings $\langle \varphi^r_{\alpha_r} ,\psi^r_{\beta_r} \rangle$ for large level $r$. We impose conditions on the sequence $\alpha_r$ and $\beta_r$: firstly, we need $\frac{\alpha_r}{r}$ and $\frac{\beta_r}{r}$ to stay in compact subsets of $\mathring{P}$ and $\mathring{Q}$. Indeed the representation of common eigenvectors of commuting Toeplitz operators works well only for eigenvalues corresponding to regular values of the principal symbols: that is, here, the interior of the polytopes. Secondly, as the eigenvectors will concentrate on Lagrangian $\Lambda_{-2 \cos (\pi \frac{\alpha_r}{r})}^{\mathcal{C}}$ and $\Lambda_{-2 \cos(\pi \frac{\beta_r}{r})}^{\mathcal{D}}$, thus we want the intersection to be transverse. By Section \ref{sec:lag}, there is an open dense subset $W_{\pitchfork}$ of $\mathring{P} \times \mathring{Q}$ such that for any $x,y \in W_{\pitchfork}$ the Lagrangian $\Lambda_{-2 \cos (\pi \frac{\alpha_r}{r})}^{\mathcal{C}}$ and $\Lambda_{-2 \cos(\pi \frac{\beta_r}{r})}^{\mathcal{D}}$ are transverse. 
\\ So we impose the following conditions on $\alpha_r$ and $\beta_r$:
\\
\\ \textbf{Property (*) :} We say that a sequence $(\alpha_r , \beta_r) \in I_r \times J_r$ satisfy Property (*) if there is a compact subset $K$ in $W_{\pitchfork}$, such that $(\frac{\alpha_r}{r},\frac{\beta_r}{r}) \in K$ for any $r$.
\\
\\ 
As explained in Section \ref{sec:holosec}, the pants decomposition $\mathcal{C}$ induces an isomorphism $\Phi_r$ from $V_r(\Sigma)$ into $H_r \subset H^0(M , L^r \otimes \delta)$, where $M=\mathbb{R}^E \times T$, where $T$ is the torus that is the fiber of the map $\mu$. 
The vectors $\psi_{\beta}^r$, as linear combinations of the $\varphi_{\alpha}^r$, can be viewed as elements of $H_r$, furthermore their $\mathbb{L}^2$-norms are concentrated near $P \times T$. 
\\ The curve operators $T_r^{D_f}$ have $\psi_{\beta}^r$ as common eigenvectors with eigenvalues $-2 \cos (\frac{\pi \beta_f}{r})$. But we have given an expression of the curve operators acting on $H_r$ as Toeplitz operators on $U=\mathring{P}\times T$. Call $\sigma^{D_f}$ the Toeplitz symbol associated to $T_r^{D_f}$, that is the trace function associated to the curve $D_f$. 
\begin{proposition}Let $W$ be a compact subset of $\mathring{P}$ and $\beta_r \in J_r$ be a sequence such that $\frac{\beta_r}{r} \in W$. The vectors $\psi_{\beta_r}^{r}$ are microlocal solutions on $U$ of
$$T_r^{\sigma^{D_f}}\Psi_r=-2 \cos( \frac{\pi \beta_{r,f}}{r})\Psi_r$$
meaning that they satisfy the following 2 conditions:
\begin{itemize}
\item[-](admissibility condition) For any compact subset $K \subset U$, there exists constants $C$ and $N$ such that $|\Psi_r(x)|\leq Cr^N$ for all $x \in K$.
\item[-](quasimode condition) For any $x \in U$ there is a function $\varphi$ with compact support containing $x$, such that $\Pi_r (\varphi \Psi_r) = \Psi_r+O(r^{-\infty})$
\\ and $T_r^{D_f} \Pi_r (\varphi \Psi_r)=-2 \cos(\frac{\pi \beta_{r,f}}{r}) \Psi_r+O(r^{-\infty})$ uniformly on a neighborhood of $x$.
\end{itemize}
Such a sequence of vectors $\Psi_r$ is also called a quasimode of the Toeplitz operators $T_r^{D_f}$, for the joint eigenvalue $-2 \cos(\frac{\pi \beta_{r,f}}{r})$
\end{proposition} 
\begin{proof}Indeed, $\psi_{\beta_r}^r$ is a linear combination of the vectors $\varphi_{\alpha}$ which are in number less or equal than $r^{|E|}$ and the coefficients in the linear combination are all less or equal than $1$ as $\psi_{\beta_r}^r$ is of norm $1$. As $h(e_{\alpha},e_{\alpha})(x) \leq 1$ for all $\alpha$ and $x \in \mathbb{R}^E$, the vectors $\psi_{\beta_r}^r$ satisfy the admissibility condition.
\\ For the quasimode condition, choose $x=(t,\theta) \in U$, let $\varphi$ be a $T$ invariant cutoff function with compact support in $U$ and identically equal to $1$ on a set of the form $V\times T$ where $V$ is a neighborhood of $x$. Then up to $O(r^{-\infty})$, the projection $\Pi_r (\varphi \psi_{\beta_r}^r)$ has the same coefficients as $\psi_{\beta_r}^r$ on each $e_{\alpha}$ with $\frac{\alpha}{r}\in V$, and on a small neighborhood $V' \subset V$ of $x$, any other $e_{\alpha}$ is $O(r^{-\infty})$. Thus $\Pi_r (\varphi \psi_{\beta_r}^r)=\psi_{\beta_r}^r +O(r^{-\infty})$ on $V'$. 
\\ Finally, we know that $\psi_{\beta_r}^r$ is an eigenvector of $T_r^{D_f}$ which on $U$ acts as a Toeplitz operator of symbol $\sigma^{D_f}$ by Theorem \ref{toeplitz}. As $\psi_{\beta_r}^r$ has the same coefficients on the $e_{\alpha}$ with $\alpha \in V$ up to $O(r^{-\infty})$, and $T_r^{D_f}$ has a finite number of nonzero diagonals, 
$$T_r^{\sigma^{D_f}}\Pi_r (\varphi \psi_{\beta_r}^r)=-2 \cos (\frac{\pi \beta_{r,f}}{r}) \varphi \psi_{\beta_r}^r+O(r^{-\infty})$$ 
on $V' \times T$.
\end{proof}
The operators $T_r^{D_f}$ of $H_r$ commute as they are curve operators on disjoint curves. 
\\ Quasimodes of commuting Toeplitz operators are well understood. When $T_1 ,\ldots ,T_n$ are commuting Toeplitz over a Kahler manifold $M$ of dimension $2 n$, with principal symbols $\mu_i$, and $E$ is a regular value of $\mu : M \rightarrow \mathbb{R}^n$, the set $\mu^{-1}(E)$ is a Lagrangian torus of $M$ by the Arnold-Liouville theorem.
\\ Quasimodes associated to eigenvalues $E^i$ concentrate on the Langrangian torus $\Lambda_E$, and some Ansatz can be used to compute the asymptotic behavior of quasimodes. We will follow the approach of \cite{Cha03}, which describes quasimodes of such operators as so-called "Lagrangian sections".
\\ Let $U \subset U'$ be two contractible neighborhoods of $E$, consisting of regular values of $\mu$. By Arnold-Liouville theorem, $\mu^{-1}(U')$ is diffeomorphic to $U' \times T$ where $T$ is an $n$-dimensional torus, and $\mu$ acts as the projection $U' \times T \rightarrow U'$ on it. For any $E' \in U'$, the torus $\Lambda_{E'}$ is Lagrangian. We take a sequence $E^r \in U$, and we are interested in quasimodes of the Toeplitz operators $T_1,\ldots ,T_n$, that is microlocal solutions of 
\begin{equation}
T_i^r \Psi_r =E_i^r \Psi_r
\label{quasimode}
\end{equation}
Proposition 3.5 \cite{Cha03} gives a formula that allows to compute quasimodes on contractible subsets in the following way:

\begin{proposition}\label{lagsection}
Let $E$ be a regular value of the momentum map $\mu :M \rightarrow \mathbb{R}^n$ and $U \subset U'$ be small contractible neighborhoods of $E$, so that $\mu^{-1}(U)=U \times T$ where $T$ is a torus. Let $V$ be a contractible open subset of the torus $T$. Then:
\begin{itemize}
\item[-] There is a smooth map $F_V :U \rightarrow \mathbb{L}^2(U'\times V, L)$ such that for any $E \in U$, the section $F_V(E)$ is flat of norm $1$ on $\Lambda_E$ and of norm $<1$ elsewhere in $U' \times V$, such that $F_V(E)$ is holomorphic in a neighborhood of $\Lambda_E$ in $U \times V$, and such that all $F_V(E)$ for $E \in U$ have the same compact support in $U' \times V$
\item[-] There is a sequence of smooth map $g_V(\cdot, r) : U \rightarrow \mathbb{L}^2(U' \times V ,\delta)$ such that $g_V(E,r)$ is holomorphic in a neighborhood of $\Lambda_E$, and $g_V(\cdot,r)$ has an asymptotic expansion $g_V(\cdot,r)=g_V^0(\cdot)+\frac{1}{r}g_V^1(\cdot) +\ldots$ with $g_V^0(E)|_{\Lambda_E}$ satisfying the transport equations $\mathcal{L}_{X_e} g_V^0(E)=0$ where $X_e$ is the symplectic gradient of the function $\mu_e$.
\end{itemize}
such that $F_V(E_r)^r g(E_r,r)$ is a microlocal solution on $U \times V$ of \eqref{quasimode}
\end{proposition} 
Furthermore, the proposition 3.6 of \cite{Cha03} garanties that quasimodes are always of this form:
\begin{proposition} \label{lagsection2} Let $E$ be a regular value of $\mu$ and $U$ and $V$ be defined as in Proposition \ref{lagsection}. Suppose that $\psi_r$ is a microlocal solution of \eqref{quasimode} on $U \times V$. Then there exists a sequence $\lambda_r$ with $\lambda_r =O(r^N)$ for some $N$ such that:
$$\psi_r = \lambda_r  F_V(E_r)^r g(E_r,r) +O(r^{-\infty})$$
\end{proposition}
\textbf{Proof of Propositions \ref{lagsection} and \ref{lagsection2}}: The material in \cite{Cha03} and \cite{Cha06} is sufficient to get these two propositions. The proof of Proposition \ref{lagsection} consists two steps: first step is a computation of how Toeplitz operators act on Lagrangian sections given by the Ansatz $\psi_r=F_V^r g(\cdot,\frac{1}{r})$. We follow the proof in \cite{Cha06} to work out this calculation.
First, the projection on $H_r$ acts on a Lagrangian section by sending $\psi_r$ to
$$\Pi_r \psi_r (x)= \int_{M} N(x,y) F^r (y) g(y) \mu_M (y)$$
where $N(x,y)$ is the Bergman kernel we computed in \ref{bergman}. What differs from \cite{Cha06} is that we integrate over a non-compact manifold $M$. However, as $[F|\leq 1$, $|g|=O(r^a)$ for some $a$, and $N(x,y) <C r^b e^{-r d(x,y)^2}$ for some constants $C$ and $a$, we can reduce this integral to an integral over a bounded open set: let $y$ be in a compact subset $K$ of $\mathring{P}\times T$, and let $\varepsilon < d(K,\partial P)$. We set $K_{\varepsilon}=\lbrace y \ /  \ d(y , K)<\varepsilon \rbrace$. We have:
$$\int_{M \setminus K_{\varepsilon}}N(x,y) F^r(y)g(y) \mu_M(y) =O(r^{-\infty})$$ 
where the $O(r^{-\infty})$ is uniform for $y \in K$. Once we reduced to an integral over a bounded set, the computations in \cite{Cha06} apply directly to show that $\Pi_r \psi_r =\psi_r +O(r^{-\infty})$ uniformly on $K$. In the same way, when we compute the action of the Toeplitz operator $T_r^{\gamma}$ on the Lagrangian section $F^r g$, we can restrict everything to an integral over a bounded set. We have that:
$$T_r^{\gamma} \psi_r (x) =\int_M N(x,y) f_{\gamma}(y) \chi(y) F^r(y) g(y) \mu(y)+O(r^{-\infty})$$
where the $O$ is uniform for $x \in K$, where $f_{\gamma}$ is the symbol of $T_r^{\gamma}$, and $\chi$ is some cutoff function that is identically $1$ over $K_{\varepsilon}$. Again, as $N(x,y) <C r^b e^{-r d(x,y)^2}$, the integral over $M \setminus K_{\varepsilon}$ is a $O(r^{-\infty})$ uniformly on $K$, and we have
$$T_r^{\gamma} \psi_r(x) =\int_{K_{\varepsilon}} N(x,y) f_{\gamma}(y) F^r (y) g(y) \mu_M(y)+O(r^{-\infty})$$
We then refer to \cite{Cha06} for the computation of the action of a Toeplitz operator on the Lagrangian section:
\\ it is shown that a Toeplitz operator of principal symbol $\mu_i$ and vanishing subprincipal symbol sends the Lagrangian section $F_V^r g$ to a Lagrangian section $$F_V^r (E_r g_0 +\frac{1}{r}(E_r g_1 +\frac{1}{i}\mathcal{L}_{X_i}g_0) +\ldots)+O(r^{-\infty})$$
These computations are again purely local, and transport directly in our setting. 
\\ Once the action of Toeplitz operators on Lagrangian sections is known, it is possible to recursively define the sections $g_i$ to get a quasimode by solving transport equations. Moreover the first term $g_0$ must satisfy $\mathcal{L}_{X_i}g_0=0$. 
\\ Then the proof of Proposition \ref{lagsection2} in \cite{Cha03} consists of using Fourier integral operators to show the microlocal equation is equivalent to an equation in a "model manifold" in which the equation can be explicitly solved. The same arguments using the control we have on the Bergman kernel to localize all integrations can be used to show that the proof in \cite{Cha03} can also be applied to our setting.$\square$
\\
\\ In general the quasimodes on such contractible open set can be patched together to get a quasimode on $U \times T$ if the sequence $E^r$ satisfy some conditions called the Bohr-Sommerfeld conditions, and then on $M$ using functions, as quasimodes are negligible away from $\Lambda_{E_r}$. Roughly speaking, the Bohr-Sommerfeld conditions consist in the following: as the sections $F_V(E_r)$ for different contractible $V \subset T$ differ by a complex number, we need to be able to renormalize them in a coherent way, this is possible when the holonomy of $L \times \delta$ along $\Lambda_{E^r}$ is trivial.
\\ In our case, we do not need to study the Bohr-Sommerfeld conditions: we already know the spectrum of $T_r^{D_f}$ and we have a sequence $\psi_{\beta_r}^r \in H_r$ that realize a quasimode on $U=\mathring{P} \times T$ of the Toeplitz operators $T_r^{D_f}$ for $E_i^r=-2 \cos(\frac{\pi \beta_{i,r}}{r})$.
\\ These quasimode have to concentrate on the Lagrangian $\Lambda_{E^r}^{D_f}$. If we chose $E^r$ in an appropriate open dense subset of $\mathbb{R}^F$, the intersection $\Lambda_{E^r}^{D_f}$ is connected according to Section \ref{sec:lag}. 
\\ Hence, if we cover $\Lambda_{E^r}^{D_f} \cap U$ by contractible open sets $V_1,\ldots ,V_k$, we know that on each $V_i$, there are coefficients $\lambda_{r,i}$ such that
$$\psi_{\beta_r}^r |_{V_i}=\lambda_{r,i} F_{V_i}^r(E^r) g_{V_i}(E^r,r)+O(r^{-\infty})$$ 
as sections in $\mathbb{L}^2(\mathring{P} \times V_i ,L^r \otimes \delta)$, and the coefficients $\lambda_{r,i}$ differ by complex numbers of norm $1$; the sections $F_{V_i}(E^r)$ can be patched together to give a section $F$ of $L$ that is flat of norm $1$ on $\Lambda_{E^r}$, similarly the $g_{V_i}(E^r ,r)$ are patched together to form a section $g$ of $\delta$. The sections $F$ and $g$ can be multivalued, however, $F^r g$ is a single-valued section of $L^r \times \delta$ which has trivial holonomy on $\Lambda_{E^r}$ as $E_r$ satisfy the Bohr-Sommerfeld conditions.
\\ Note that here we have used the fact that the intersection $U \cap \Lambda_{E^r}$ is connected, otherwise we would need multiple constants $\lambda_r$, one for each connected component of the intersection.
\\
\\
Now that we know an asymptotic expression of $\psi_{\beta_r}^r$ as a Lagrangian section, we want to calculate $|\lambda_r|$ using the fact that $\psi_{\beta_r}^r$ is of norm $1$.
\begin{proposition}The vectors $\psi_{\beta_r}$ for $\frac{\beta_r}{r} \in W$ where $W$ is a compact subset of $\mathring{Q}$, have an asymptotic expansion as elements of $H_r$:
\begin{displaymath} \psi_{\beta_r}=u_r \left(\frac{r}{2 \pi} \right)^{\frac{n}{4}}(1+O(r^{-1}))F^r(E^r)g(E^r) 
\end{displaymath}
where 
\begin{itemize}
\item[-]$u_r$ is a sequence of complex number of moduli $1$
\item[-]$E^r$ is the sequence of common eigenvalues corresponding to $\psi_{\beta_r}$, given by the formula $E^r_i=-2 \cos (\pi \frac{\beta_{r,i}}{r})$
\item[-]$F$ and $g(\cdot,r)$ are the smooth maps in $W \rightarrow \mathbb{L}^2 (M,L)$ and $W \rightarrow \mathbb{L}^2 (M,\delta)$ respectively, with $F(E^r)$ flat of norm $1$ on $\Lambda_{E^r}^{\mathcal{D}}$, holomorphic in a neighborhood of this Lagrangian, and of norm $<1$ elsewhere.
\\ And finally the sections $g(E^r)$ of $\delta$ on $M$ are holomorphic in a neighborhood of $\Lambda_{E^r}^{\mathcal{D}}$ and have an asymptotic expansion $g=g_0+\frac{1}{r}g_1+\ldots$ with $g_0$ a smooth section of the half-form bundle $\delta$ such that on $\Lambda_{E^r}^{\mathcal{D}}$ we have $g_0^{\otimes 2}=\frac{1}{\mathrm{Vol}(\Lambda_{E^r}^{\mathcal{D}})}d \theta_1 \wedge \ldots d \theta_n$, where the $\theta_i$ are angle coordinates on $\Lambda_{E^r}^{\mathcal{D}}$.
\end{itemize} 
\end{proposition}
\begin{proof}The norm of Lagrangian sections are easily computed using stationary phase lemma: according to \cite{Cha03}, the Lagrangian section is normalized when we normalize the section $g_V$ by $g_V^0(E)^{\otimes 2}=\frac{1}{\mathrm{Vol}(\Lambda_{E^r})}d\theta_1 \wedge \ldots \wedge d \theta_n$ where $\theta_i$ are angle coordinates on $\Lambda_{E^r}$, and $\lambda_r=\left(\frac{r}{2\pi}\right)^{\frac{n}{4}}$.
\\ But as the difference between $\psi_{\beta_r}^r$ and the Lagrangian section is $O(r^{-\infty})$ uniformly only on compact subset of $U=\mathring{P}\times T$, we need to be careful that $\psi_{\beta_r}$ does not carry too much weight over small neighborhoods of $\partial P \times T$. 
\\ To compute the coefficient $\lambda_r$, we will introduce an operator $\Pi$ to localize our eigenvector on $\mathring{P}\times T$. Let $\rho$ be a cutoff function with compact support inside $\mathring{P}$ and identically equal to $1$ on the open set $\lbrace x \in P \, / \, d(x,\partial P)>\varepsilon \rbrace$. For $\varepsilon$ sufficiently small, this set has non trivial intersection with $\Lambda_{-2 \cos(\pi x)}^{\mathcal{D}}$ for any $x$ in a given compact subset $W$ of $\mathring{Q}$.
\\ Now the operator $\Pi_r \in \textrm{End}(V_r(\Sigma))$ acts on the basis of $\varphi_{\alpha}$ associated to the pants decomposition $\mathcal{C}$ by: 
\begin{displaymath}\Pi_r \varphi_{\alpha}=\rho (\frac{\alpha}{r})\varphi{\alpha}
\end{displaymath}
The operator $\Pi_r$ can actually be seen as a function of the curve operators $T_r^{C_e}$: we have $\Pi_r=\rho (\frac{1}{\pi}\textrm{acos}(-\frac{T_r^{C_e}}{2}))$.
 Therefore, it can be viewed as a Toeplitz operator of principal symbol
  $\sigma = \rho (\frac{1}{\pi}\textrm{acos}(-\frac{f_{C_e}}{2}))$
   and vanishing subprincipal symbol, where $f_{C_e}$ are the trace functions associated to the curves $C_e$. In \cite{Cha03}, it is stated that for a Lagrangian section 
 $\Psi_r=\lambda_r (\frac{r}{2\pi})^{\frac{n}{4}}F^r g$ 
 concentrating on $\Lambda_{E^r}$, and for a Toeplitz operator $T_r$ of principal symbol $\sigma$ and vanishing subprincipal symbol, we have 
\begin{displaymath}
\langle T_r \Psi_r , \Psi_r \rangle = (1+O(r^{-1}))|\lambda_r|^2 \frac{1}{\mathrm{Vol}(\Lambda_{E^r})} \int_{\Lambda_{E^r}} \sigma d\theta_1 \wedge \ldots d \theta_n 
\end{displaymath}
We can apply this for the vector $\psi_{\beta_r}$ as an element of $H_r$ and the Toeplitz operator $\Pi_r$ to get an expression of $\langle \Pi_r \psi_{\beta_r} , \psi_{\beta_r} \rangle$.
\\ But instead of using the isomorphism from $V_r(\Sigma)$ to $H_r$ corresponding to the pants decomposition $\mathcal{C}$, we also have an isomorphism corresponding to the decomposition $\mathcal{D}$. With this isomorphism, $\psi_{\beta_r}$ is sent to a monomial $e_{\beta_r}$ in $M'= \mathbb{R}^F \times T ' $, which is a Lagrangian section concentrating on $\lbrace \frac{\beta_r}{r}\rbrace\times T'$, furthermore, in this simple situation, the coefficient $\lambda_r$ is exactly $(\frac{r}{2 \pi})^{\frac{n}{4}}$. Though the operator $\Pi_r$ does not have a simple expression as a diagonal operator in the base of the $e_{\beta}$, it is still a Toeplitz operator of principal symbol $\sigma =\rho (\frac{1}{\pi}\textrm{acos}(-\frac{f_{C_e}}{2}))$ and vanishing subprincipal symbol in this new setting. Hence we have:
 \begin{displaymath}
\langle \Pi_r \psi_{\beta_r} , \psi_{\beta_r} \rangle = (1+O(r^{-1}))(\frac{r}{2 \pi})^{\frac{n}{2}} \frac{1}{\mathrm{Vol}(T')} \int_{\lbrace \frac{\beta_r}{r} \rbrace \times T'} \sigma d\theta_1 \wedge \ldots d \theta_n 
\end{displaymath}
\\ Comparing the two asymptotic expansions of $\langle \Pi_r \psi_{\beta_r} , \psi_{\beta_r} \rangle$, as the integral of the principal symbol of $\Pi_r$ on $\Lambda_{E^r}$ is non-vanishing, we get that the coefficient of normalization is indeed that of the proposition.
\end{proof}
\subsection{A formula for pairings of eigenvectors}
\label{sec:formula}
We are ready to prove our final theorem:
\begin{theo}\label{pairing} Let $\mathcal{C}$ and $\mathcal{D}$ be two pair of pants decompositions of a closed oriented surface $\Sigma$. Let $\alpha_r$ and $\beta_r$ r-admissible colorings for the two pants decompositions such that $(\alpha_r , \beta_r)$ satisfies Property (*), and let $\varphi_{\alpha_r}$ and $\psi_{\beta_r}$ be the corresponding basis vectors of $V_r(\Sigma)$. Then we have the following asymptotic expansion:
 \begin{displaymath}\langle \varphi_{\alpha_r} ,\psi_{\beta_r} \rangle = u_r\left(\frac{r}{2\pi}\right)^{-\frac{n}{2}}\frac{1}{\sqrt{\mathrm{Vol}(\Lambda^{\mathcal{C}}_{E_{\alpha_r}})\mathrm{Vol}(\Lambda^{\mathcal{D}}_{E_{\beta_r}'})}}\underset{z \in \Lambda^{\mathcal{C}}_{E_{\alpha_r}} \cap \Lambda^{\mathcal{D}}_{E_{\beta_r}'}}{\sum}\frac{e^{i r \eta (z)}i^{m(z)}}{ |\det(\lbrace \mu_i , \mu_j ' \rbrace)|^{\frac{1}{2}}}+O(r^{-\frac{n}{2}-1})
 \end{displaymath}
 where $n=3 g-3$ is half the dimension of the moduli space, $u_r$ is a sequence of complex numbers of moduli $1$, $\mu_i=-\tr (\rho (C_i))$ (resp. $\mu_j '=-\tr (\rho (D_j))$ ) are the principal symbols of the curve operators $T_r^{C_i}$ (resp. $T_r^{D_j}$), the volumes of the Lagrangians are volumes for $n$-forms dual to the $n$-vectors $X_1 \wedge \ldots \wedge X_n$ (resp. $X_1 ' \wedge \ldots \wedge X_n '$) of Hamiltonian vector fields of $\mu_i$ (resp. $\mu_j '$ ), $e^{i \eta(z)} \in \mathbb{U}$ is the holonomy of $L$ along a loop $\gamma_{z_0 , z}$ which goes from a reference point  $z_0 \in \Lambda^{\mathcal{C}}_{E_{\alpha_r}} \cap \Lambda^{\mathcal{D}}_{E_{\beta_r}'}$ to $z$ in $\Lambda^{\mathcal{C}}_{E_{\alpha_r}} $ then back to $z_0$ in $\Lambda^{\mathcal{D}}_{E_{\beta_r}'}$ and finally $m(z) \in \mathbb{Z}$ corresponds to a Maslov index.
 \end{theo}
\begin{proof} Let $\mathcal{C}$ and $\mathcal{D}$ be two pants decompositions of $\Sigma$, and $(\alpha_r ,\beta_r)$ be a sequence of admissible $r$-colorings (that is index of basis vectors) satisfying Property (*) of \ref{sec:quasimodes}.
\\ We consider pairings of the vectors $\varphi_{\alpha_r}$ and $\psi_{\beta_r}$. The first is a common eigenvector of the $T_r^{C_e}$ with eigenvalues $E^r_e=-2 \cos (\frac{\pi \alpha_{r,e}}{r})$.
 The second is a common eigenvector of the $T_r^{D_f}$
  with common eigenvalues $E'^r_f =-2 \cos (\frac{\pi \beta_{r,f}}{r})$.
\\ We use the first pants decomposition as a decomposition of reference, giving us an isomorphism $\Phi_r$ between $V_r(\Sigma)$ and a space $H_r$ of holomorphic sections of a complex line bundle $L^r \otimes \delta$, by the work done in Section \ref{isom}. Under this isomorphism, we know that the images of the vectors $\psi_{\beta_r}$ are Lagrangian sections, concentrating on the Lagrangian $\Lambda^{\mathcal{D}}_{E'^r}$, and of the form $(1+O(r^{-1})(\frac{r}{2\pi})^{\frac{n}{4}}F^r g$, where $F$ section of $L$ and $g$ section of $\delta$ satisfying the conditions explained in Section \ref{sec:quasimodes}. The same is true for the $\varphi_{\alpha_r}$. (actually, the situation is even simpler, as the isomorphism $\Phi_r$ sends the vectors $\varphi_{\alpha_r}$ to the vectors $e_{\alpha_r}$ of $H_r$, which are exactly the expected Lagrangian sections).  
\\ As these sections concentrate respectively on $\Lambda_{E^r}^{\mathcal{C}}$ and $\Lambda_{E'^r}^{\mathcal{D}}$, the only meaningful contribution in the integral comes from the intersection points of these two Lagrangians. But as $\alpha_r$ and $\beta_r$ were carefully chosen to respect Property (*), the intersection of these two Lagrangians is always transversal, in particular, consists of a finite set of points. The contribution of each intersection point can be computed by means of stationary phase methods, the computations are done in \cite{Cha03}.
\\ For two Lagrangian sections $(\frac{r}{2\pi})^{\frac{n}{4}}F_1 g_1$ and $(\frac{r}{2\pi})^{\frac{n}{4}}F_2 g_2$ concentrating on $\Lambda_1$ and $\Lambda_2$, the first order of the contribution of an intersection point $z$ of their Lagrangian supports is $(\frac{r}{2\pi})^{-\frac{n}{2}}F_1 (z)^r \overline{F_2} (z)^r(g_1(z),g_2(z))_{T_z \Lambda_1 , T_z \Lambda_2}$.
\\ But $F_1$ and $F_2$ are flat of norm $1$ on $\Lambda_1$ and $\Lambda_2$. Thus, if we write $(F_1\overline{F_2})^r (z)=e^{i r \eta (z)}$ and pick a point $z_0 \in \Lambda_1 \cap \Lambda_2$ of reference, then $e^{i \eta (z)-\eta(z_0)}$ is the holonomy of the line bundle $L$ along a loop $\gamma_{z_0,z}$ which goes from $z_0$ to $z$ in $\Lambda_1$ and returns from $z$ to $z_0$ in $\Lambda_2$.
\\ Furthermore, up to normalization, $g_1^{\otimes 2}(z)$ and $g_2^{\otimes 2}(z)$ are $n$-forms dual to the $n$-vectors $X_1 \wedge \ldots \wedge X_n (z)$ and $X_1 ' \wedge \ldots \wedge X_n '(z)$ (where the vector fields $X_i$ and $X_i'$ are Hamiltonian vector fields of $\mu_i$ and $\mu_i '$). Thus the pairing $(g_1 (z),g_2(z))_{T_z \Lambda_1 , T_z \Lambda_2}$ is a square root of $\frac{1}{\mathrm{Vol}(\Lambda_1)\mathrm{Vol}(\Lambda_2)}\textrm{det}(\lbrace \mu_i , \mu_j ' \rbrace)^{-1}(z)$.
\\ We can introduce integers $m(z)$ such that
$$(g_1(z) , g_2 (z) )_{T_z \Lambda_1 , T_z \Lambda_2}=\frac{1}{\sqrt{\mathrm{Vol}(\Lambda_1)\mathrm{Vol}(\Lambda_2)}}|\textrm{det}(\lbrace \mu_i , \mu_j ' \rbrace)|^{-\frac{1}{2}}(z) i^{m(z)}$$
\\ Recall that $g_1$ and $g_2$ are sections of $\delta$ such that $g_1^{\otimes 2}$ and $g_2^{\otimes 2}$ are the
 complexification of the $n$-forms on $\Lambda_1$ and $\Lambda_2$ given by $\frac{1}{\mathrm{Vol}(\Lambda_1)}\beta_1$ and 
 $\frac{1}{\mathrm{Vol}(\Lambda_2)}\beta_2$ (where $\beta_1$ and $\beta_2$ are dual to the Hamiltonian vector fields of the two 
 sets of principal symbols). The pairings of these sections have been described in \ref{sec:halfform}, which gives a rule 
 depending on the relative positions of the Lagrangian $\Lambda_1$ and $\Lambda_2$ to choose the square root. 
\\ These definitions of $\eta(z)$ and $m(z)$ depend only on the homotopy class $\gamma_{z_0 ,z}$ as $L$ and $\delta$ are flat on $\Lambda_1$ and $\Lambda_2$. Furthermore, $L^r \otimes \delta$ is flat and trivial on $\Lambda_1$ and $\Lambda_2$ as Bohr-Sommerfeld conditions are verified, thus the asymptotic expansion does not depend on the choice of $\gamma_{z_0 ,z}$ at all. 
\end{proof}
\subsection{A geometric interpretation of the phase and index}
\label{geominter}
 Our theorem \ref{pairing} introduces two quantities: a phase $\eta (z)$ and an integer index $m(z)$ where $z$ is in the intersection of the two Lagrangian $\Lambda_1$ and $\Lambda_2$ of the theorem. They are defined using features of Kähler geometry: holonomy of a prequantizing bundle, parallel transport in a half-form bundle, and the pairings in half-form bundle of Section \ref{sec:halfform}. From this description, the procedure to compute the index $m(z)$ seems rather intricate. 
 \\ We would like a simple geometric picture to interpret both the phase $\eta(z)$ and the index $m(z)$. Of course, only their variation are relevant: if we choose a reference point $z_0$ in the intersection $\Lambda_1 \cap \Lambda_2$, we can assume $\eta(z_0)=0$ and $m(z_0)=0$ just by changing the moduli $1$ complex number $u_r$ appearing in Theorem \ref{pairing}. The geometric picture we have in mind should preferably involve only the symplectic geometry and not the complex structure $J$ on our quantizing space $M$, as it is the only structure inherited from the moduli space $\mathcal{M}(\Sigma)$. 
 \\ 
 \\ An interesting case is when a loop $\gamma_{z_0 ,z}$ is trivial in $\pi_1 M$ and thus bounds a disk $D_{z_0 ,z}$. As the 
prequantizing line bundle $L$ has curvature $\frac{\omega}{i}$, the holonomy of $L$ along $\gamma_{z_0 ,z }$ is the same as
 $e^{i A(D_{z_0,z})}$, where $A(D_{z_0 ,z})$ is the symplectic area of $D_{z_0 , z}$. 
\\ As for the index $m(z)$, observe that as $D(z_0 ,z)$ is contractile, the half-form bundle $\delta$ is trivial on it. After choosing $g_1(z_0)$ and $g_2 (z_0)$ (for which there is a sign ambiguity), the value of $g_1(x)$ and $g_2(x)$ is determined for any $x \in D(z_0 ,z)$ by parallel transport. View $g_1(z_0)^{\otimes 2}$ as the complexification of a $n$-form on $\Lambda_1$, then following $\gamma_{z_0 ,z}$ we get a path $e(x)$ in the oriented Lagrangian Grassmanian $LG^+ (D(z_0,z))$ of $D(z_0,z)$, such that for $x \in \gamma_{z_0 ,z}$, the element $g_1(x)^{\otimes 2}$ is the complexification of a positive $n$-form on $e(x)$. The same can be done for the return map from $z$ to $z_0$ in $\Lambda_2$, we get a path $f$. There is a canonical way to connect these two paths to get a loop in the Lagrangian Grassmanian. Indeed, fix a Lagrangian frame  $L$. The set of Lagrangians $L'$ transverse to $L$ is affine: any such Lagrangian is the graph of a map $A: L \rightarrow J L$ such that $JA$ is symmetric, thus defines a quadratic form on $L$. We can thus connect $L'$ to $J L$ by a segment. This give us paths $p_{z_0}$ and $p_z$ from $T_{z_0} \Lambda_2$ to $J T_{z_0} \Lambda_1$ and from $J T_z \Lambda_1$ to $T_z \Lambda_2$. The path $J e$ allows us to close the path $p_z f p_{z_0}$. We get a close path in the oriented Lagrangian Grassmanian, the $\pi_1$ class of which is exactly $m(z)$. 
\\ Indeed, as our index $m(z)$ and the class we defined depend only on the Lagrangian $\Lambda_1$ and $\Lambda_2$, we can move our Lagrangian so that at points of intersection $z_0$ and $z$ we have $T \Lambda_2= J T \Lambda_1$. Then the pairing $(\cdot ,\cdot )_{T \Lambda_1 , T \Lambda_2 }$ in the half-form bundle is positively proportional to the square root of the Hermitian pairing on $n$-form on $D_{z_0 ,z}$. Thus, the ambiguity in the square root comes from which square roots of $d \theta_1 \wedge \ldots \wedge d \theta_n$ (resp $d \theta_1 ' \wedge \ldots \wedge d \theta_n '$) the section $g_1$ (resp. $g_2$) represents at $z_0$ and $z$. It is easy to see that following a loop of class $1$ in $\pi_1 LG^+ (D_{z_0 ,z})$, parallel transport changes $g_i$ by a $-$ sign. Hence the $\pi_1$-class of $(J e) \cdot p_z \cdot f \cdot p_{z_0}$ calculates the index $m(z)$
\\ The definition of the index seems to depend on the quasi-complex structure $J$. However, the set of quasi-complex structures on the disk $D_{z_0 , z}$ is affine. As the index we defined depend continuously on $J$, it must be constant when the quasi-complex structure $J$ varies.  
 \\
 \\ The argument to interpret geometrically the index $m(z)$ works only when the loop $\gamma_{z_0 ,z}$ bounds a disk in $\mathring{P}\times T$. The situation is more complicated when the loop is not trivial in $\mathring{P} \times T$ (whose fundamental group is the same as $T$, that is $\mathbb{Z}^n$), and the interpretation of the index is not clear in this picture, and seems to depend on our specific choice of half-form bundle.
  \\ A possible way of tackling this problem would be to show that our choice of half-form bundle derives from the choice of a spin-structure on $\mathcal{M}(\Sigma)$. The loop $\gamma_{z_0 ,z}$ can be defined as a loop in $\mathcal{M}(\Sigma)$. As the fundamental group of the moduli space $\mathcal{M}(\Sigma)$ is trivial (as explained in \cite{RSW}), this loop always bounds a disk in $\mathcal{M}(\Sigma)$. We expect the index $m(z)$ to be computable as the class of some specific loop in the Lagrangian Grassmanian of $\mathcal{M}(\Sigma)$. We leave such considerations to a following paper.

\end{document}